\documentclass[a4paper,11pt]{amsart}

\usepackage {amssymb}
\usepackage{amsmath}
\usepackage{amsthm}
\usepackage[all]{xy}

\theoremstyle{plain}
\newtheorem{theorem}{Theorem}[section]
\newtheorem{lemma}[theorem]{Lemma}

\newtheorem{proposition}[theorem]{Proposition}

\theoremstyle{definition}
\newtheorem{definition}[theorem]{Definition}
\newtheorem{remarks}[theorem]{Remarks}
\newtheorem{remark}[theorem]{Remark}

\def\ad{\mathop{\hbox {ad}}\nolimits}
\def\Cas{\mathop{\hbox {Cas}}\nolimits}
\def\dim{\mathop{\hbox {dim}}\nolimits}
\def\End{\mathop{\hbox {End}}\nolimits}
\def\rank{\mathop{\hbox {rank}}\nolimits}
\def\sgn{\mathop{\hbox {sgn}}\nolimits}
\def\str{\mathop{\hbox {str}}\nolimits}
\def\tr{\mathop{\hbox {tr}}\nolimits}

\numberwithin{equation}{section}

\begin{document}
\title[The Kostant criterion for Lie superalgebras]{An analogue of the Kostant criterion for quadratic Lie superalgebras}

\author{Yifang Kang}
\address{Institute of Mathematics and Physics, Central South University of Forestry and Technology, Changsha Hunan, 410004, P. R. China} \email{kangyf1978@sina.com}
\author{Zhiqi Chen}
\address{School of Mathematical Sciences and LPMC, Nankai University, Tianjin 300071, P.R. China} \email{chenzhiqi@nankai.edu.cn}

\begin{abstract}
Assume that $\mathfrak{r}$ be a finite dimensional complex Lie superalgebra with a non-degenerate super-symmetric invariant bilinear form, $\mathfrak{p}$ is a finite dimensional complex super vector space with a non-degenerate super-symmetric bilinear form, and $\nu: \mathfrak{r}\rightarrow\mathfrak{osp}(\mathfrak{p})$ is a homomorphism of Lie superalgebras. In this paper, we give a necessary and sufficient condition for $\mathfrak{r}\oplus\mathfrak{p}$ to be a quadratic Lie superalgebra. The criterion obtained is an analogue of a constancy condition given by Kostant in Lie algebra setting. As an application, we prove an analogue of the Parthasarathy's formula for the square of the Dirac operator attached to a pair of quadratic Lie superalgebras.
\end{abstract}



\date{\today}

\maketitle


\setcounter{section}{0}

\section{Introduction}
 A {\it quadratic Lie superalgebra} is a  Lie superalgebra $\mathfrak{g}=\mathfrak{g}_{\bar 0}\oplus \mathfrak{g}_{\bar 1}$ with a non-degenerate invariant super-symmetric bilinear form $(\cdot,\cdot)$. We always assume that $(\cdot,\cdot)$ is {\it consistent}, that is, $(x,y)=0$ for any $x\in \mathfrak{g}_{\bar 0}$ and $y\in \mathfrak{g}_{\bar 1}$. Let $\mathfrak{r}$ be a subalgebra of a finite dimensional complex quadratic Lie superalgbra $\mathfrak{g}$ such that the restriction of $(\cdot,\cdot)$ on $\mathfrak{r}$ is non-degenerate. Denote by $\mathfrak{p}$ the orthogonal complement of $\mathfrak{r}$ in $\mathfrak g$ with respect to $(\cdot,\cdot)$. Then we have an orthogonal decomposition $\mathfrak{g}=\mathfrak{r}\oplus\mathfrak{p}$, where the restriction of $(\cdot,\cdot)$ on $\mathfrak{p}$ is also non-degenerate and $\mathfrak{p}$ is an $\mathfrak{r}$-module.

Conversely, let $\mathfrak{r}$ be a finite dimensional complex quadratic Lie superalgebra with respect to a bilinear form $(\cdot,\cdot)_\mathfrak{r}$, let $\mathfrak{p}$ be a finite dimensional complex super vector space with a non-degenerate super-symmetric bilinear form $(\cdot,\cdot)_\mathfrak{p}$, and let \begin{equation*}
\nu :\mathfrak{r}\rightarrow\mathfrak{osp}(\mathfrak{p})
\end{equation*}
be a $(\cdot,\cdot)_\mathfrak{p}$-invariant representation of $\mathfrak{r}$ on $\mathfrak{p}$. Define
\begin{equation*}
\mathfrak{g}=\mathfrak{r}\oplus\mathfrak{p}
\end{equation*}
and define a non-degenerate super-symmetric bilinear form $(\cdot,\cdot)_\mathfrak{g}$ on $\mathfrak{g}$ by
$$(\cdot,\cdot)_\mathfrak{g}|_\mathfrak{r}=(\cdot,\cdot)_\mathfrak{r},\quad (\cdot,\cdot)_\mathfrak{g}|_\mathfrak{p}=(\cdot,\cdot)_\mathfrak{p},\quad (\mathfrak{p},\mathfrak{r})_\mathfrak{g}=0.$$ 
The pair $(\nu, (\cdot,\cdot)_\mathfrak{g})$ is of {\it Lie super type} if there exists a Lie superalgebraic structure $[\cdot,\cdot]$ on $\mathfrak{g}$ satisfying the following conditions:
\begin{enumerate}
\item[(a)] $\mathfrak{g}$ is a quadratic Lie superalgebra with respect to $(\cdot,\cdot)_\mathfrak{g}$, and
\item[(b)] $\mathfrak{r}$ is a subalgebra of $\mathfrak{g}$ and $[x,y]=\nu (x)y$ for any $x\in\mathfrak{r},y\in\mathfrak{p}$.
\end{enumerate}
In \cite{Ko2,Ko3}, Kostant studied the above problem in Lie algebra setting and obtained a constancy condition involving the Casimir element of $\mathfrak{r}$ and a cubic element in $(\Lambda^3(\mathfrak{p}))^\mathfrak{r}$ which is used to construct the cubic Dirac operator. In this paper, we obtain an analogue of a constancy condition for the case of Lie superalgebras based on the study in \cite{CK}.

We begin with the case $\mathfrak{r}=0$. For this case, it is to find all quadratic Lie superalgebraic structures on a complex super vector space with a non-degenerate super-symmetric bilinear form. Clearly, for any quadratic Lie superalgebra $\mathfrak{g}$, there exists a unique $\phi\in\Lambda_{\bar 0}^3\mathfrak{g}$ such that
\begin{equation*}
(\phi,z_1\wedge z_2\wedge z_3)=-\frac{1}{2}([z_1,z_2],z_3), \quad [z_1,z_2]=2\iota(z_1)\iota(z_2)\phi.
\end{equation*}
Motivated by the above fact, for any $\phi\in\Lambda_{\bar 0}^3\mathfrak{g}$, define a bracket $[\cdot,\cdot]^\phi$ on $\mathfrak{g}$ by
\begin{equation*}
[z_1,z_2]^\phi=2\iota(z_1)\iota(z_2)\phi.
\end{equation*}
We prove that the bracket $[\cdot,\cdot]^\phi$ defines a Lie superalgebraic structure on $\mathfrak{g}$ if and only if the Clifford square $\phi^2$ is a constant.

In the general case, let $\phi_\mathfrak{r}\in\Lambda_{\bar 0}^3(\mathfrak{r})$ be the cubic element corresponding to the quadratic Lie superalgebraic structure on $\mathfrak{r}$, and let $\phi_\mathfrak{p}\in \Lambda_{\bar 0}^3(\mathfrak{p})$ be the cubic element given as the projection of $\phi$ relative to the decomposition $\mathfrak{g}=\mathfrak{r}\oplus\mathfrak{p}$. If $(\nu,(\cdot,\cdot)_\mathfrak{g})$ is of Lie super type, then the cubic element $\phi$ is decomposed as
\begin{equation}\label{eq I.1}
\phi=\phi_\mathfrak{r}+\phi_\mathfrak{p}+\sum_{1\leq i\leq r}\nu_*(x_i)\wedge x^i,
\end{equation}
where $\{x_1,\ldots,x_r\}$ is a basis of $\mathfrak{r}$, $\{x^1,\ldots,x^r\}$ is the $(\cdot,\cdot)_\mathfrak{r}$-dual basis to $\{x_1,\ldots,x_r\}$, and \begin{equation*}
\nu_* :\mathfrak{r}\rightarrow\Lambda^2(\mathfrak{p})
\end{equation*}
is the unique Lie superalgebraic homomorphism induced by $\nu$. Moreover, $\phi_\mathfrak{p}\in (\Lambda_{\bar 0}^3(\mathfrak{p}))^\mathfrak{r}$. Conversely, for any $\phi_\mathfrak{p}\in (\Lambda_{\bar 0}^3(\mathfrak{p}))^\mathfrak{r}$, define the cubic element $\phi$ by (\ref{eq I.1}). We prove that $\phi^2$ is a scalar if and only of $\nu_*(\Cas_\mathfrak{r})+\phi^2_\mathfrak{p}$ is a constant, thus $(\nu,(\cdot,\cdot)_\mathfrak{g})$ is of Lie super type if and only if $\nu_*(\Cas_\mathfrak{r})+\phi^2_\mathfrak{p}$ is a constant.

This paper is organized as follows. In Section 2, we recall some basic facts about Clifford algebras and exterior algebras over super vector spaces. Sections 3 and 4 are to study the case $\mathfrak{r}=0$ and the general case, respectivley. As an application, we prove an analogue of the Parthasarathy's formula for the square of the Dirac operator attached to a pair of quadratic Lie superalgebras in Section 5.

\section{Preliminaries}
\subsection{Super vector spaces}
A $\mathbb{Z}_2$-graded space $V=V_{\bar 0}+V_{\bar 1}$ is called a \emph{super vector space}, where the elements of $V_{\bar 0}$ are even and those of $V_{\bar 1}$ are odd. Denote by $|x|\in\{{\bar 0},{\bar 1}\}$ the parity of a homogeneous element $x\in V$. (Whenever this notation is used, it implies that $x$ is homogeneous.) We say that a bilinear form $(\cdot,\cdot)$ on $V$ is super-symmetric if $(x,y)=(-1)^{|x||y|}(y,x)$ for any $x,y\in V$; consistent if $(V_{\bar 0},V_{\bar 1})=0$. Throughout this paper, we always assume that $(\cdot,\cdot)$ is consistent, that is,
\begin{equation}\label{eq 2.1.1}
   (x,y)=0,\quad \text{if}\ |x|\neq |y|.
\end{equation}

For a finite dimensional super vector space $V$, let $\{e_1,\ldots,e_m\}$ of $V_{\bar 0}$ be a basis of $V_{\bar 0}$ and $\{e_{m+1},\ldots,e_{m+n}\}$ a basis of $V_{\bar 1}$. Corresponding to the homogeneous basis $\{e_1,\ldots,e_{m+n}\}$ of $V$, the matrix of an endomorphism $T$ on $V$ is the form $\left(\begin{array}{cc}
\alpha& \beta\\
\gamma& \delta\\
\end{array}\right)$, where $\alpha$ is an $(m\times m)$-, $\beta$ an $(m\times n)$-, $\gamma$ an $(n\times m)$-, and $\delta$ an $(n\times n)$-matrix. Define \emph{the supertrace} $\str(T)$ of $T$ by $$\str(T)=\tr(\alpha)-\tr(\delta).$$ It is clear that $\str(T)$ is independent of the choice of a homogeneous basis.


\begin{lemma}\label{lemma 2.1}
Let $V$ be a finite dimensional super vector space with a non-degenerate super-symmetric bilinear form $(\cdot,\cdot)$, let ${\mathcal A}$ be an associative algebra, and let $f,g: V\rightarrow {\mathcal A}$ be two linear mappings. Assume that $\{x_1,\ldots,x_n\}$ is a homogeneous basis of $V$ and $\{x^1,\ldots,x^n\}$ is the $(\cdot,\cdot)$-dual basis of $\{x_1,\ldots,x_n\}$. Then $\sum_{i=1}^n f(x_i)g(x^i)$ is independent of the choice of basis. In particular,
\begin{equation*}
\sum_{i=1}^n f(x_i)g(x^i)=\sum_{i=1}^n (-1)^{|x_i||x^i|}f(x^i)g(x_i).
\end{equation*}
\end{lemma}
\begin{proof}
Let $\{y_1,\ldots,y_n\}$ be another basis of $V$ and let $\{y^1,\ldots,y^n\}$ be the $(\cdot,\cdot)$-dual basis of $\{y_1,\ldots,y_n\}$. Let $S=(s_{ij})$ and $T=(t_{ij})$ be $n\times n$ matrices satisfying
$$ (y_1,\ldots,y_n)=(x_1,\ldots,x_n)S, \quad  (y^1,\ldots,y^n)=(x^1,\ldots,x^n)T,$$
that is, $y_i=\sum\limits_{j=1}^ns_{ji}x_j$ and $y^i=\sum\limits_{j=1}^nt_{ji}x^j$ for any $i=1,\ldots,n$. Since
\begin{equation*}
\delta_{ij}=(y^i,y_j)=(\sum_{k=1}^nt_{ki}x^k,\sum_{k=1}^ns_{kj}x_k)=\sum_{k=1}^ns_{kj}t_{ki},
\end{equation*}
we have that $S^TT=E_n$, which implies that $TS^T=E_n$, that is, $\sum\limits_{i=1}^nt_{ki}s_{li}=\delta_{kl}$. Here $E_n$ is the $n\times n$ identity matrix. Now,
\begin{equation*}
\sum_{i=1}^n f(y_i)g(y^i)=\sum_{i=1}^n\sum_{k=1}^n\sum_{l=1}^ns_{li}t_{ki}f(x_l)g(x^k)=\sum_{k=1}^n\sum_{l=1}^n\delta_{kl}f(x_l)g(x^k)=\sum_{i=1}^nf(x_i)g(x^i),
\end{equation*}
which implies that $\sum_{i=1}^n f(x_i)g(x^i)$ is independent of the choice of basis. The last statement follows from the fact that $\{(-1)^{|x_i||x^i|}x_i\}$ is the $(\cdot,\cdot)$-dual basis of $\{x_1,\ldots,x_n\}$.
\end{proof}

A \emph{superalgebra} is a super vector space $\mathcal{A}=\mathcal{A}_{\bar 0}+\mathcal{A}_{\bar 1}$ with a multiplication satisfying $\mathcal{A}_i\mathcal{A}_i\subset\mathcal{A}_{i+j}$ for any $i,j\in\mathbb{Z}_2$. For superalgebras $\mathcal{A}$ and $\mathcal{B}$, $\mathcal{A}\otimes\mathcal{B}$ is a superalgebra with the multiplication defined by
\begin{equation}\label{eq 2.1.2}
(x\otimes y)(x^\prime\otimes y^\prime)=(-1)^{|y||x^\prime|}xx^\prime\otimes yy^\prime.
\end{equation}

A \emph{Lie superalgebra} $\mathfrak{g}$ is a superalgebra $\mathfrak{g}=\mathfrak{g}_{\bar 0}\oplus\mathfrak{g}_{\bar 1}$ with a bracket $[\cdot,\cdot]$ satisfying
\begin{gather}
  [x,y]=-(-1)^{|y||x|}[y,x],      \label{eq 2.1.3}       \\
  [x,[y,z]]=[[x,y],z]+(-1)^{|y||x|}[y,[x,z]]. \label{eq 2.1.4}
\end{gather}
Here the identity (\ref{eq 2.1.3}) is the skew super-symmetry and the identity (\ref{eq 2.1.4}) is the super Jacobi identity. For more details on Lie superalgebras, see \cite{Ka}. The $\mathbb{Z}_2$-gradation of $V$ induces
\begin{equation*}
\End (V)=\End (V)_{\bar 0}\oplus\End (V)_{\bar 1},
\end{equation*}
where
\begin{equation*}
\End (V)_i=\{\xi\in\End (V)|\xi(V_j)\subset V_{i+j}\}
\end{equation*}
for any $i,j\in\mathbb{Z}_2$. It is easy to see that $\mathfrak{gl}(V)=\End (V)$ is a Lie superalgebra under the commutator defined by
\begin{equation*}
[\xi_1,\xi_2]=\xi_1\xi_2-(-1)^{|\xi_1||\xi_2|}\xi_2\xi_1,\quad \forall\xi_1,\xi_2\in\End (V).
\end{equation*}
It is called the \emph{general linear Lie superalgebra over $V$}. A \emph{representation} $\rho$ of a Lie superalgebra $\mathfrak{g}$ on $V$ is a homomorphism $\rho: \mathfrak{g}\rightarrow\mathfrak{gl}(V)$ of Lie superalgebras which preserves the grading. Note that the map $\ad:\mathfrak{g}\rightarrow\mathfrak{gl}(\mathfrak{g})$ is a representation of $\mathfrak{g}$, where $\ad x(y)=[x,y]$ for any $x,y\in\mathfrak{g}$. It is called the \emph{adjoint representation}.

Let $(\cdot,\cdot)$ be a non-degenerate super-symmetric bilinear form on $V$. Then
\begin{equation*}
\mathfrak {osp}(V)=\{\delta\in\End (V)|(\delta(x),y)+(-1)^{|\delta||x|}(x,\delta(y))=0\}
\end{equation*}
is a subalgebra of $\mathfrak{gl}(V)$, which is called the \emph{ortho-symplectic Lie superalgebra} over $V$ with respect to $(\cdot,\cdot)$. A bilinear form $(\cdot,\cdot)$ on a Lie superalgebra $\mathfrak{g}$ is called \emph{invariant} if
\begin{equation*}
([x,y],z)=(x,[y,z]), \quad\forall x,y,z\in\mathfrak{g}.
\end{equation*}
A Lie superalgebra $\mathfrak g$ together with a non-degenerate invariant super-symmetric bilinear form $(\cdot,\cdot)$ is called a \emph{quadratic Lie superalgebra}.
Let $\{x_1,\ldots,x_n\}$ be a basis of the quadratic Lie superalgebra ${\mathfrak g}$ and let $\{x^1,\ldots,x^n\}$ be the $(\cdot,\cdot)$-dual basis of $\{x_1,\ldots,x_n\}$. By Lemma \ref{lemma 2.1}, $$\Cas_{\mathfrak g}=\sum_{i=1}^n x_ix^i \in U({\mathfrak g})$$ is independent of the choice of basis. It is the Casimir element of ${\mathfrak g}$. Moreover, $\Cas_{\mathfrak g}$ belongs to the center $Z({\mathfrak g})$ of the enveloping algebra $U({\mathfrak g})$ of ${\mathfrak g}$.

\subsection{Clifford algebras and exterior algebras over super vector spaces} Let $V$ be a finite dimensional super vector space with a non-degenerate super-symmetric bilinear form $(\cdot,\cdot)$. This subsection is to recall some facts on the Clifford algebra and the exterior algebra over $V$. For more details on Clifford theory, see \cite{CK,Ko1,Me}.

Let $T(V)$ be the tensor algebra over $V$. Denote by $I_C(V)$ (resp. $I_\Lambda(V)$) the ideal in $T(V)$ generated by all elements of the form, for any $x,y\in V$,  $$x\otimes y+(-1)^{|x||y|}y\otimes x-2(x,y)\quad (\textrm{resp. } x\otimes y+(-1)^{|x||y|}y\otimes x).$$ Then we have the \emph{Clifford algebra} $C(V)=T(V)/I_C(V)$ (resp. the \emph{exterior algebra} $\Lambda(V)=T(V)/I_\Lambda(V)$). Composing the canonical injection $V\rightarrow T(V)$ with the quotient mapping $\pi_C:T(V)\rightarrow C(V)$ (resp. $\pi_\Lambda:T(V)\rightarrow \Lambda(V)$), we obtain the canonical mapping $$\zeta_{C}:V\rightarrow C(V)\quad (\textrm{resp. } \zeta_\Lambda:V\rightarrow \Lambda(V)).$$ Moreover, we may identify $V$ with $\zeta_{C}(V)$ (resp. $\zeta_\Lambda(V)$) so that $C(V)$ (resp. $\Lambda(V)$) is the algebra generated  by $V$ with the relation $$xy+(-1)^{|x||y|}yx=2(x,y) \quad (\textrm{ resp. } x\wedge y+(-1)^{|x||y|}y\wedge x=0),$$ where $xy$ (resp. $x\wedge y$) is the Clifford multiplication of $C(V)$ (resp. the exterior multiplication of $\Lambda(V)$) for any $x,y\in V$.
The pair $(C(V),\zeta _{C})$ (resp. $(\Lambda(V),\zeta _{C})$) has the following standard universal mapping property.
\begin{proposition}\label{prop 2.2}
Assume that  ${\mathcal A}$ is an associative algebra with the unity element $1_{\mathcal A}$ and $\phi :V\rightarrow {\mathcal A}$ is a linear mapping such that
\begin{equation*}
\phi (x)\phi(y)+(-1)^{|x||y|}\phi (y)\phi (x)=2(x,y)1_{\mathcal A}\quad (resp.\ \phi (x)\phi(y)+(-1)^{|x||y|}\phi (y)\phi (x)=0)
\end{equation*}
for any $x,y\in V$. Then $\phi$ extends uniquely to an algebra homomorphism $\phi _{C} :C(V)\rightarrow {\mathcal A}$ (resp. $\phi _{\Lambda} :\Lambda(V)\rightarrow {\mathcal A}$).
\end{proposition}

It is well-known that $T(V)$ has a natural ${\mathbb Z}\times {\mathbb Z}_2$-gradation. The degree of $x_1\otimes\cdots\otimes x_n$ is equal to $(n, |x_1|+\cdots+|x_n|)$. Since $C(V)$ and $\Lambda(V)$ inherit the ${\mathbb Z}_2$-gradation from $T(V)$, we still denote by $|u|$ the parity of a homogeneous element $u$ in $T(V)$ (resp. $C(V)$, $\Lambda(V)$). The ${\mathbb Z}$-gradation of $T(V)$ induces a ${\mathbb Z}$-gradation of $\Lambda(V)$, but only induces a ${\mathbb Z}_2$-gradation of $C(V)$.

Denote by $T^n_a(V)$ the subspace spanned by the elements of degree $(n,a)$ in $T(V)$. Then
\begin{equation*}
T(V)=\bigoplus_{n\in{\mathbb Z},a\in{\mathbb Z}_2} T^n_a(V).
\end{equation*}
Set $T^n(V)=\sum_{a\in{\mathbb Z}_2} T^n_a(V)$ and $T_a(V)=\sum_{n\in{\mathbb Z}}T^n_a(V)$.
\begin{definition}\label{def 2.3}
A linear mapping $D:T(V)\rightarrow T(V)$ is called a \emph{derivation of degree $(k,d)$} if
\begin{enumerate}
\item[(i)] $D(T^n(V))\subset T^{k+n}(V)$ and $D(T_a(V))\subset T_{a+d}(V)$,
\item[(ii)] $D(u\otimes v)=D(u)\otimes v+(-1)^{kn}(-1)^{da}u\otimes D(v), \forall u\in T^n_a(V),v\in T(V)$.
\end{enumerate}
\end{definition}
Similarly, one can define the derivation of $C(V)$ and $\Lambda(V)$. If $D_T$ is a derivation of $T(V)$ which stabilizes both $I_C(V)$ and $I_\Lambda(V)$, then $D_T$ descends a derivation $D_C$ of $C(V)$ and a derivation $D_\Lambda$ of $\Lambda(V)$.

For any homogeneous element $x\in V$, there is a unique derivation $\iota_T (x)$ of $T(V)$ such that $\iota_T (x)(y)=(x,y)$ for any $y\in V$. Explicitly,
\begin{equation*}
 \iota_T (x)(x_1\otimes\cdots\otimes x_n)\\
=\sum_{k=1}^n(-1)^{k-1}(-1)^{|x|(|x_0|+\cdots +|x_{k-1}|)}(x,x_k)x_1\otimes\cdots\otimes\widehat{x_k}\otimes\cdots\otimes x_n,
\end{equation*}
where $x_1,\ldots,x_n\in V$ and $|x_0|={\bar 0}$.
Clearly, $\iota_T (x)$ is a derivation of degree $(-1,|x|)$ by the identity (\ref{eq 2.1.1}). By Proposition 4.5 in \cite{CK}, $\iota_T (x)$ stabilizes both $I_C(V)$ and $I_\Lambda(V)$. Then $\iota_T (x)$ descends to derivations $\iota_{C} (x)$ and $\iota_{\Lambda}(x)$ of $C(V)$ and $\Lambda(V)$, respectively.

For any $x\in V$, let $\epsilon_{\Lambda}(x)$ be the left exterior multiplication operator by $x$ on $\Lambda(V)$. By Proposition 4.6 in \cite{CK}, we have
\begin{gather}
\epsilon_{\Lambda} (x)\epsilon_{\Lambda}(y)+(-1)^{|x||y|}\epsilon_{\Lambda}(y)\epsilon_{\Lambda} (x)=0, \label{eq 2.2.1}\\
\iota_{\Lambda} (x)\iota_{\Lambda}(y)+(-1)^{|x||y|}\iota_{\Lambda}(y)\iota_{\Lambda} (x)=0, \label{eq 2.2.2}\\
\iota_{\Lambda} (x)\epsilon_{\Lambda}(y)+(-1)^{|x||y|}\epsilon_{\Lambda}(y)\iota_{\Lambda} (x)=(x,y).           \label{eq 2.2.3}
\end{gather}
Set $\gamma (x)=\epsilon_{\Lambda}(x)+\iota_{\Lambda} (x)$. Then
\begin{equation}\label{eq 2.2.4}
\gamma (x)\gamma (y)+(-1)^{|x||y|}\gamma (y)\gamma (x)=2(x,y).
\end{equation}
The linear map $V\rightarrow \End (\Lambda(V))$ defined by
$x\mapsto\gamma (x)$ naturally extends to a homomorphism
$T(V)\rightarrow \End (\Lambda(V))$, which, by the identity (\ref{eq 2.2.4}), descends to a homomorphism
\begin{equation*}
\gamma : C(V)\rightarrow \End (\Lambda(V)).
\end{equation*}
The homomorphism $\gamma$ defines a $C(V)$-module structure on $\Lambda(V)$. Let $\eta : C(V)\rightarrow \Lambda(V)$ be the linear map defined by
\begin{equation*}
\eta(u)=\gamma(u)1_{\Lambda(V)},
\end{equation*}
where $1_{\Lambda(V)}$ is the unity element of $\Lambda(V)$.

Define the \emph{skew super symmetrization map} $s:\Lambda(V)\rightarrow T(V)$ by
\begin{equation*}
s(x_1\wedge x_2\wedge\cdots\wedge x_n)=\frac{1}{n!}\sum_{\sigma\in S_n}(-1)^{N_\sigma(x_1,\ldots,x_n)}\sgn(\sigma)x_{\sigma (1)}\otimes x_{\sigma
(2)}\otimes\cdots\otimes x_{\sigma (n)}
\end{equation*}
for homogeneous elements $x_1,\ldots,x_n\in V$. Here $\sgn(\sigma)$ denotes the signature of the permutation $\sigma$, and $N_\sigma(x_1,\ldots,x_n)$ is the number of pairs $i<j$ such that $x_i, x_j$ are odd elements and $\sigma^{-1}(i)>\sigma^{-1}(j)$. Let $\Sigma(V)=s(\Lambda(V))$ be the space of skew super-symmetric tensors. Set $\hat{\pi}_C=\pi_C|_{\Sigma(V)}$ and $\hat{\pi}_\Lambda=\pi_\Lambda|_{\Sigma(V)}$.
\begin{lemma}[\cite{CK}, Theorem 4.8]\label{lemma 2.4}
The map $\eta$ is bijective. Moreover, one has
\begin{equation}
\eta\circ\hat{\pi}_C=\hat{\pi}_\Lambda,
\end{equation}
that is, the following diagram commutes.
\center $\xymatrix{
    & \Sigma(v)\ar[dl]_{\hat{\pi}_C}\ar[dr]^{\hat{\pi}_\Lambda} &\\
 C(V)\ar[rr]^{\eta} & &\Lambda(V)
}$
\end{lemma}
Using $\eta$, we may identify $C(V)$ with $\Lambda(V)$. There exist two multiplications on $\Lambda(V)$, that is, the exterior multiplication $u\wedge v$ and the Clifford multiplication $uv$.

Set $\Lambda^n(V)=\pi_{\Lambda}(T^n(V))$ and denote the component of $u$ in $\Lambda^n(V)$ by $(u)_n$ for any $u\in\Lambda(V)$. One can extends the bilinear form $(\cdot,\cdot)$ on $V$ to a non-degenerate bilinear form on $\Lambda(V)$, which is still denote by $(\cdot,\cdot)$,
\begin{equation*}
(u,v)=\left\{ \begin{array}{ll} (-1)^{\frac{n(n-1)}{2}}(uv)_0,&m=n,\\0,&m\neq n, \end{array}\right.
\end{equation*}
for any $u\in \Lambda^m(V)$ and $v\in \Lambda^n(V)$, where we identify $\Lambda^0(V)$ with $\mathbb C$. If $u=x_1\wedge\cdots\wedge x_n$ is an element in $\Lambda^n(V)$, then
\begin{equation*}
(u,v)=(-1)^\frac{n(n-1)}{2}\iota_\Lambda (x_1)\cdots\iota_\Lambda (x_n)v
\end{equation*}
for any $v\in \Lambda^n(V)$. Moreover, we have
\begin{lemma}[\cite{CK}, Theorem 5.4]\label{lemma 2.5} Let $x\in V$ and $u,v\in \Lambda(V)$. Then
\begin{enumerate}
\item[(i)] $(u,v)=(-1)^{|u||v|}(v,u)$.
\item[(ii)] $(\epsilon_\Lambda(x)u,v)=(-1)^{|x||u|}(u,\iota_\Lambda(x)v)$.
\item[(iii)] $(\iota_\Lambda(x)u,v)=(-1)^{|x||u|}(u,\epsilon_\Lambda(x)v)$.
\end{enumerate}
\end{lemma}
Every $u\in\Lambda^2(V)$ defines an operator $\ad u$ on $\Lambda(V)$:
\begin{equation*}
\ad u(v)=[u,v]_C=uv-(-1)^{|u||v|}vu,\quad\forall v\in\Lambda(V).
\end{equation*}
It is proved in \cite{CK} that $\ad u$ is a derivation of degree $(0, |u|)$ of $\Lambda(V)$. Moreover, we have
\begin{equation}\label{eq 2.2.4-1}
\ad u(z)=-2(-1)^{|u||z|}\iota_{\Lambda}(z)u
\end{equation}
for any $z\in V$ and
\begin{equation}\label{eq 2.2.4-2}
(\ad u(v_1),v_2)+(-1)^{|u||v_1|}(v_1,\ad u(v_2))=0
\end{equation}
for any $v_1,v_2\in\Lambda(V)$. Now define a map
\begin{equation*}
A:\ \Lambda^2(V)\rightarrow\mathfrak{osp} (V,\varepsilon)
\end{equation*}
by $A(u)=(\ad u)|_V$.
\begin{lemma}[\cite{CK}, Theorem 6.3]\label{lemma 2.6}
$\Lambda^2(V)$ is a Lie superalgebra under the commutator
\begin{equation*}
[u,v]_C=uv-(-1)^{|u||v|}vu,
\end{equation*}
and the map $A:\ \Lambda^2(V)\rightarrow\mathfrak{osp} (V,\varepsilon)$ is an isomorphism.
\end{lemma}

Let $\xi_T\in\End (T(V))$. Suppose that both $I_{\Lambda}(V)$ and $I_{C}(V)$ are
stable under $\xi_T$. Then $\xi_T$ descends to a map $\xi_{\Lambda}$ (resp.
$\xi_{C}$) of $\Lambda(V)$ (resp. $C(V)$).
\begin{lemma}[\cite{CK}, Lemma 4.10]\label{lemma 2.7}
If $\Sigma(V)$ is stable under $\xi_T$, then $\eta\circ\xi_{C}=\xi_{\Lambda}\circ\eta$ on $C(V)$.
\end{lemma}
If $I_{\Lambda}(V)$, $I_{C}(V)$ and $\Sigma(V)$ are stable under $\xi_T$, then by Lemma \ref{lemma 2.7}, we may identify $\xi_{C}$ with $\xi_{\Lambda}$ on $\Lambda(V)$ by means of $\eta$.
\begin{lemma}\label{lemma 2.8}
For any homogeneous element $x\in V$, $\iota_T (x)$ stabilizes $\Sigma(V)$.
\end{lemma}
\begin{proof}
Let $\sigma\in S_n$ be a permutation of the set $\{1,\ldots,n\}$. For any integer $l$ between $1$ and $n$, there exists a unique integer $k$ such that $\sigma(k)=l$. Denote the permutation group of the set of $\{1,\ldots,n\}\backslash\{l\}$ by $S_{n,l}$. Let $\tau(\sigma;k,l)\in S_{n,l}$ be the permutation such that
\begin{equation*}
\tau(\sigma;k,l)(1,\ldots,\hat{l},\ldots,n)=(\sigma(1),\ldots,\widehat{\sigma(k)},\ldots,\sigma(n)).
\end{equation*}
Recall that $N_\sigma(x_1,\ldots,x_n)$ is the number of pairs $i<j$ such that $x_i, x_j$ are odd elements and $\sigma^{-1}(i)>\sigma^{-1}(j)$. Then
\begin{equation*}
\begin{aligned}
 x_1\wedge\cdots\wedge x_n=&(-1)^{N_\sigma(x_1,\ldots,x_n)}\sgn(\sigma)x_{\sigma (1)}\wedge\cdots \wedge x_{\sigma (n)}\\
=&(-1)^{k-1+N_\sigma(x_1,\ldots,x_n)}(-1)^{|x_{\sigma (k)}|(|x_{\sigma (1)}|+\cdots +|x_{\sigma (k-1)}|)}\sgn(\sigma) \\
&x_{\sigma (k)}\wedge x_{\sigma (1)}\wedge\cdots \wedge \widehat{x_{\sigma(k)}}\wedge\cdots \wedge x_{\sigma (n)}.
\end{aligned}
\end{equation*}
On the other hand, we have
\begin{equation*}
\begin{aligned}
 x_1\wedge\cdots\wedge x_n=&(-1)^{l-1}(-1)^{|x_l|(|x_1|+\cdots +|x_{l-1}|)}x_l\wedge x_1\wedge\cdots\wedge \widehat{x_l}\wedge\cdots\wedge x_n\\
=&(-1)^{l-1+N_{\tau(\sigma;k,l)}(x_1,\ldots,\widehat{x_l},\ldots,x_n)}(-1)^{|x_l|(|x_1|+\cdots +|x_{l-1}|)}\sgn(\tau(\sigma;k,l)) \\
& x_l\wedge x_{\tau(\sigma;k,l)(1)}\wedge\cdots\wedge \widehat{x_{\tau(\sigma;k,l)(l)}}\wedge\cdots\wedge x_{\tau(\sigma;k,l)(n)}\\
=&(-1)^{l-1+N_{\tau(\sigma;k,l)}(x_1,\ldots,\widehat{x_l},\ldots,x_n)}(-1)^{|x_l|(|x_1|+\cdots +|x_{l-1}|)}\sgn(\tau(\sigma;k,l))\\
&x_{\sigma (k)}\wedge x_{\sigma (1)}\wedge\cdots \wedge \widehat{x_{\sigma(k)}}\wedge\cdots \wedge x_{\sigma (n)},
\end{aligned}
\end{equation*}
which implies that
\begin{equation}\label{eq 2.2.5}
\begin{aligned}
 &(-1)^{k-1+N_\sigma(x_1,\ldots,x_n)}(-1)^{|x_{\sigma (k)}|(|x_{\sigma (1)}|+\cdots +|x_{\sigma (k-1)}|)}\sgn(\sigma)\\
=&(-1)^{l-1+N_{\tau(\sigma;k,l)}(x_1,\ldots,\widehat{x_l},\ldots,x_n)}(-1)^{|x_l|(|x_1|+\cdots +|x_{l-1}|)}\sgn(\tau(\sigma;k,l)).
\end{aligned}
\end{equation}

Conversely, for any $\tau\in S_{n,l}$, there exist permutations $\sigma_i(1\leq i\leq n)$ of $\{1,\ldots,n\}$ such that
\begin{equation}\label{eq 2.2.6}
\tau(\sigma_i;i,l)=\tau.
\end{equation}
In fact, $\sigma_i$ is defined by $\sigma_i(i)=l$ and
\begin{equation*}
\sigma_i(1,\ldots,\hat{i},\ldots,n)=(\tau(1),\ldots,\widehat{\tau(l)},\ldots,\tau(n)).
\end{equation*}

Let $S(k;x_1,x_2,\ldots,x_k)$ denote
\begin{equation*}
  \sum_{\sigma\in S_k}(-1)^{N_\sigma(x_1,\ldots,x_k)}\sgn(\sigma)x_{\sigma (1)}\otimes\cdots\otimes x_{\sigma (k)}.
\end{equation*}
Identities (\ref{eq 2.2.5}) and (\ref{eq 2.2.6}) imply that
\begin{equation*}
\begin{aligned}
 &\iota_T(x)(S(n;x_1,x_2,\ldots,x_n))\\
=&\iota_T(x)(\sum_{\sigma\in S_n}(-1)^{N_\sigma(x_1,\ldots,x_n)}\sgn(\sigma)x_{\sigma (1)}\otimes\cdots\otimes x_{\sigma (n)})\\
=&\sum_{\sigma\in S_n}\sum_{k=1}^n (-1)^{k-1+N_\sigma(x_1,\ldots,x_n)}(-1)^{|x_{\sigma (k)}|(|x_{\sigma (1)}|+\cdots +|x_{\sigma (k-1)}|)}\sgn(\sigma) \\
&(x,x_{\sigma (k)}) x_{\sigma (1)}\otimes\cdots \otimes \widehat{x_{\sigma(k)}}\otimes\cdots \otimes x_{\sigma (n)}\\
=&\sum_{\sigma\in S_n}\sum_{l=1}^n (-1)^{l-1+N_{\tau(\sigma;\sigma^{-1}(l),l)}(x_1,\ldots,\widehat{x_l},\ldots,x_n)}(-1)^{|x_l|(|x_1|+\cdots +|x_{l-1}|)}\sgn(\tau(\sigma;\sigma^{-1}(l),l))\\
 &(x,x_l)x_{\tau(\sigma;\sigma^{-1}(l),l)(1)}\otimes\cdots\otimes \widehat{x_{\tau(\sigma;\sigma^{-1}(l),l)(l)}}\otimes\cdots\otimes x_{\tau(\sigma;\sigma^{-1}(l),l)(n)}\\
=&n\sum_{l=1}^n\sum_{\tau\in S_{n,l}}(-1)^{l-1+N_{\tau}(x_1,\ldots,\widehat{x_l},\ldots,x_n)}(-1)^{|x_l|(|x_1|+\cdots +|x_{l-1}|)}\sgn(\tau)(x,x_l)\\
 &x_{\tau(1)}\otimes\cdots\otimes \widehat{x_{\tau(l)}}\otimes\cdots\otimes x_{\tau(n)}\\
=&n\sum_{l=1}^n(-1)^{l-1}(-1)^{|x_l|(|x_1|+\cdots +|x_{l-1}|)}(x,x_l)S(n-1;x_1,\ldots,\widehat{x_l},\ldots,n).
\end{aligned}
\end{equation*}
The lemma follows.
\end{proof}
Thus, we may identify $\iota_{\Lambda}$ with $\iota_C$ on $\Lambda(V)$ by means of $\eta$, and we denote $\iota_{\Lambda}(=\iota_C)$ by $\iota$.
Let $\alpha_T$ be the linear map of $T(V)$ defined by
\begin{equation*}
\alpha_T(x_1\otimes\cdots\otimes x_n)=(-1)^{\frac{n(n-1)}{2}+N_{\sigma_0}(x_1,\ldots,x_n)}\sgn(\sigma)(x_n\otimes\cdots\otimes x_1),
\end{equation*}
where $x_1,x_2,\ldots,x_k$ are homogeneous elements in $V$ and
\begin{equation*}
\sigma_0=\left(
          \begin{array}{cccc}
            1&2&\ldots&n \\
            n&n-1&\ldots&1 \\
          \end{array}
        \right).
\end{equation*}
It is proved in \cite{CK} that $\alpha_T$ stabilizes $I_{\Lambda}(V)$, $I_{C}(V)$ and $\Sigma(V)$.
Hence, $\alpha_T$ descends to the linear map $\alpha_{\Lambda}$ (resp. $\alpha_{C}$) of $\Lambda(V)$ (resp. $C(V)$); we may identify $\alpha_{\Lambda}$ with $\alpha_{C}$ on $\Lambda(V)$ by means of $\eta$, and we denote $\alpha_{\Lambda}(=\alpha_{C})$ by $\alpha$.
\begin{lemma}[\cite{CK}]\label{lemma 2.10}
The linear map $\alpha$ has the following properties.
\begin{enumerate}
\item[(i)] $\alpha^2=1$.
\item[(ii)] $\alpha(u)=(-1)^{\frac{n(n-1)}{2}}u$ for any $u\in\Lambda^n(V)$.
\item[(iii)] $\alpha(u\wedge v)=(-1)^{|u||v|}\alpha(v)\wedge \alpha(u)$ and $\alpha(uv)=(-1)^{|u||v|}\alpha(v)\alpha(u)$ for any $u,v\in\Lambda(V)$.
\end{enumerate}
\end{lemma}

\begin{lemma}\label{lemma 2.11}
Let $x,u$ be homogeneous elements in $V$ and $\Lambda^k(V)$ respectively. Then
\begin{equation*}
x u+(-1)^{k-1}(-1)^{|x||u|}u x=2\iota(x)u.
\end{equation*}
\end{lemma}
\begin{proof}
Recall that $\gamma(x)=\iota(x)+\epsilon(x)$ is the operator of left Clifford multiplication in $\Lambda(V)$ by $x$. Let $\gamma_R(x)\in\End \Lambda(V)$ be the operator of right Clifford multiplication in $\Lambda(V)$ by $x$. Then
\begin{equation*}
x u+(-1)^{k-1}(-1)^{|x||u|}u x=\gamma(x)(u)+(-1)^{k-1}(-1)^{|x||u|}\gamma_R(x)(u).
\end{equation*}
Let
\begin{equation*}
\End^j(\Lambda(V))=\{\xi\in\End (\Lambda(V))|\xi(\Lambda^i(V))\subset \Lambda^{i+j}(V)\}.
\end{equation*}
Then
\begin{equation*}
\End(\Lambda(V))=\bigoplus_{j\in\mathbb{Z}}\End^j(\Lambda(V)).
\end{equation*}
Write $\gamma(x)=\sum\limits_{j\in {\mathbb Z}}a_{j}$ and $\gamma_R(x)=\sum\limits_{j\in {\mathbb Z}}b_{j}$, where $a_{j},b_{j}\in \End ^{j}(\Lambda(V))$. Then $a_{j}=0$ unless $j\in\{-1,1\}$, and $a_{-1}=\iota(x), a_1=\epsilon(x)$. By Lemma \ref{lemma 2.10},
\begin{equation*}
\gamma_R(x)(u)=u x=\alpha(\alpha(u x))=(-1)^{\frac{k(k-1)}{2}}(-1)^{|x||u|}\alpha(x u)=(-1)^{\frac{k(k-1)}{2}}(-1)^{|x||u|}\alpha(\gamma(x)u).
\end{equation*}
It follows that
\begin{equation*}
b_j(u)=(-1)^{\frac{k(k-1)}{2}}(-1)^{|x||u|}\alpha(a_j(u))=(-1)^{\frac{j(2k+j-1)}{2}}(-1)^{|x||u|}a_j(u).
\end{equation*}
Hence we have $b_{j}=0$ unless $j\in\{-1,1\}$, and
\begin{equation*}
b_{-1}(u)=(-1)^{k-1}(-1)^{|x||u|}a_{-1}(u),\quad b_1=(-1)^k(-1)^{|x||u|}a_1(u).
\end{equation*}
Therefore, we have
\begin{equation*}
x u+(-1)^{k-1}(-1)^{|x||u|}u x=2a_{-1}(u)=2\iota(x)u.
\end{equation*}
The lemma follows.
\end{proof}

\section{Quadratic Lie superalgebraic structures on super vector spaces}\label{section3}
Let $\mathfrak{g}$ be a finite dimensional complex super vector space with a non-degenerate super-symmetric bilinear form $(\cdot,\cdot)$.

If there exists a Lie superalgebraic structure $[\cdot,\cdot]$ on $\mathfrak{g}$ such that $\mathfrak{g}$ is quadratic with respect to $(\cdot,\cdot)$, then there exists a unique $\phi\in\Lambda_{\bar 0}^3\mathfrak{g}$ such that
\begin{equation*}
(\phi,z_1\wedge z_2\wedge z_3)=-\frac{1}{2}([z_1,z_2],z_3)
\end{equation*}
for any $z_1,z_2,z_3\in\mathfrak{g}$.

Conversely, for any $\phi\in\Lambda_{\bar 0}^3\mathfrak{g}$, define the bracket $[\cdot,\cdot]^\phi$ on $\mathfrak{g}$ by
\begin{equation*}
[z_1,z_2]^\phi=2\iota(z_1)\iota(z_2)\phi.
\end{equation*}
First, by the identity (\ref{eq 2.2.2}), $[\cdot,\cdot]^\phi$ has skew super-symmetry, that is,
\begin{equation*}
[z_1,z_2]^\phi=-(-1)^{|z_1||z_2|}[z_2,z_1]^\phi
\end{equation*}
for any $z_1,z_2\in\mathfrak{g}$. Next, by Lemma \ref{lemma 2.5}, we have
\begin{equation*}
([z_1,z_2]^\phi,z_3)=(2\iota(z_1)\iota(z_2)\phi,z_3)=-2(\phi,z_1\wedge z_2\wedge z_3)
\end{equation*}
and
\begin{equation*}
(z_1,[z_2,z_3]^\phi)=(z_1,2\iota(z_2)\iota(z_3)\phi)=-2(\phi,z_1\wedge z_2\wedge z_3).
\end{equation*}
for homogeneous elements $z_1,z_2,z_3\in\mathfrak{g}$. It follows that $$([z_1,z_2]^\phi,z_3)=(z_1,[z_2,z_3]^\phi),$$ that is, $(\cdot,\cdot)$ is invariant with respect to the bracket $[\cdot,\cdot]^\phi$. Finally, we will give the condition for the bracket $[\cdot,\cdot]^\phi$ satisfying the super Jacobi
identity. Denote the Clifford square of $u$ by $u^2$ for any $u\in\Lambda(\mathfrak{g})$. By Lemma \ref{lemma 2.10}, we have $\alpha(\phi^2)=(-1)^{|\phi||\phi|}\alpha(\phi)\alpha(\phi)=\phi^2$. Then $$\phi^2=(\phi^2)_4+(\phi^2)_0.$$
\begin{lemma}\label{lemma III.1}
If $\phi\in\Lambda_{\bar 0}^3\mathfrak{g}$, then
\begin{equation}\label{eq r10}
\iota(z_1)\iota(z_2)\iota(z_3)\phi^2=\frac{1}{2}([z_1,[z_2,z_3]^\phi]^\phi-[[z_1,z_2]^\phi,z_3]^\phi-(-1)^{|z_1||z_2|}[z_2,[z_1,z_3]^\phi]^\phi)
\end{equation}
for homogeneous elements $z_1,z_2,z_3\in\mathfrak{g}$.
\end{lemma}
\begin{proof}
By a direct calculation, we have
\begin{equation*}
\begin{aligned}
 &\iota(z_1)\iota(z_2)\iota(z_3)\phi^2\\
=&((\iota(z_1)\phi)(\iota(z_2)\iota(z_3)\phi))-(-1)^{|z_1||z_2|+|z_1||z_3|}(\iota(z_2)\iota(z_3)\phi)(\iota(z_1)\phi))\\
 &-(-1)^{|z_1||z_2|}((\iota(z_2)\phi)(\iota(z_1)\iota(z_3)\phi)-(-1)^{|z_1||z_2|+|z_2||z_3|}(\iota(z_1)\iota(z_3)\phi)(\iota(z_2)\phi))\\
 &+(-1)^{|z_1||z_3|+|z_2||z_3|}((\iota(z_3)\phi)(\iota(z_1)\iota(z_2)\phi)-(-1)^{|z_1||z_3|+|z_2||z_3|}(\iota(z_1)\iota(z_2)\phi)(\iota(z_3)\phi)).
\end{aligned}
\end{equation*}
Note that $\iota(z_2)\iota(z_3)\phi=\frac{1}{2}[z_2,z_3]^\phi$, by the identity (\ref{eq 2.2.4-1}), we have
\begin{equation*}
\begin{aligned}
 &((\iota(z_1)\phi)(\iota(z_2)\iota(z_3)\phi))-(-1)^{|z_1||z_2|+|z_1||z_3|}(\iota(z_2)\iota(z_3)\phi)(\iota(z_1)\phi))\\
=&\frac{1}{2}\ad(\iota(z_1)\phi)([z_2,z_3]^\phi)=-(-1)^{|z_1||z_2|+|z_1||z_3|}\iota([z_2,z_3]^\phi)\iota(z_1)\phi=\frac{1}{2}[z_1,[z_2,z_3]^\phi]^\phi.
\end{aligned}
\end{equation*}
Similarly, we have
\begin{equation*}
(\iota(z_2)\phi)(\iota(z_1)\iota(z_3)\phi)-(-1)^{|z_1||z_2|+|z_2||z_3|}(\iota(z_1)\iota(z_3)\phi)(\iota(z_2)\phi)=\frac{1}{2}[z_2,[z_1,z_3]^\phi]^\phi
\end{equation*}
and
\begin{equation*}
(\iota(z_3)v)(\iota(z_1)\iota(z_2)\phi)-(-1)^{|z_1||z_3|+|z_2||z_3|}(\iota(z_1)\iota(z_2)\phi)(\iota(z_3)\phi)=\frac{1}{2}[z_3,[z_1,z_2]^\phi]^\phi.
\end{equation*}
Hence
\begin{equation*}
\iota(z_1)\iota(z_2)\iota(z_3)\phi^2=\frac{1}{2}([z_1,[z_2,z_3]^\phi]^\phi-[[z_1,z_2]^\phi,z_3]^\phi-(-1)^{|z_1||z_2|}[z_2,[z_1,z_3]^\phi]^\phi).
\end{equation*}
That is, the lemma follows.
\end{proof}
By the above lemma, the bracket $[\cdot,\cdot]^\phi$ satisfies the super Jacobi identity if and only if $(\phi^2)_4=0$, thus $\phi^2=(\phi^2)_0$ is a constant.
In summary, we have the following theorem.
\begin{theorem}\label{theorem III.2}
Let $\mathfrak{g}$ be a finite dimensional complex super vector space with a non-degenerate super-symmetric bilinear form $(\cdot,\cdot)$. Then there is a one-to-one correspondence between the set $$\mathcal{V}=\{\phi\in\Lambda_{\bar 0}^3\mathfrak{g}|\phi^2\in\mathbb{C}\}$$ and the set of quadratic Lie superalgebraic structures on $\mathfrak{g}$ with respect to $(\cdot,\cdot)$. Here the bracket of the Lie superalgebra $\mathfrak g$ corresponding to $\phi\in\mathcal{V}$ is defined by
\begin{equation*}
[z_1,z_2]^{\phi}=2\iota(z_1)\iota(z_2)\phi, \quad \forall z_1,z_2\in\mathfrak{g}.
\end{equation*}
\end{theorem}
\begin{remarks}
The set $\mathcal{V}$ is not empty, since it at least contains the zero element. In the case $\phi=0$, the corresponding quadratic Lie superalgebra is super-commutative.
\end{remarks}

The following is to determine the scalar $\phi^2$ for any $\phi\in\mathcal{V}$. Let $\mathfrak{g}$ be the Lie superalgebra under the bracket  $$[z_1,z_2]^{\phi}=2\iota(z_1)\iota(z_2)\phi.$$ Suppose $\{e_1,\ldots,e_n\}$ is a homogeneous basis of $\mathfrak{g}$, with the $(\cdot,\cdot)$-dual homogeneous basis $\{e^1,\ldots,e^n\}$ of $\mathfrak{g}$. Here $|e_i|=|e^i|$ for $1\leq i \leq n$. Then
\begin{equation*}
\{e_i\wedge e_j\wedge e_k|1\leq i\leq j\leq k\leq n,e_i\wedge e_j\wedge e_k\neq 0\}
\end{equation*}
is a basis of $\Lambda^3\mathfrak{g}$, with the dual basis
\begin{equation*}
\left\{\frac{1}{(e^i\wedge e^j\wedge e^k,e_i\wedge e_j\wedge e_k)}e^i\wedge e^j\wedge e^k|1\leq i\leq j\leq k\leq n,e^i\wedge e^j\wedge e^k\neq 0\right\}.
\end{equation*}
In terms of these bases,
\begin{equation*}
\phi=\sum_{1\leq i\leq j\leq k\leq n,\atop e_i\wedge e_j\wedge e_k\neq 0}\frac{(\phi,e_i\wedge e_j\wedge e_k)}{(e^i\wedge e^j\wedge e^k,e_i\wedge e_j\wedge e_k)}e^i\wedge e^j\wedge e^k.
\end{equation*}
Since $\phi\in\Lambda_{\bar 0}^3\mathfrak{g}$, it follows that $$(\phi,e_i\wedge e_j\wedge e_k)= 0 \textit{ if } |e_i|+|e_j|+|e_k|\neq{\bar 0}.$$ If $i,j,k$ are different from each other and $|e_i|+|e_j|+|e_k|={\bar 0}$, then
\begin{equation*}
(e^i\wedge e^j\wedge e^k,e_i\wedge e_j\wedge e_k)=(-1)^{|e^i||e^j|+|e^i||e^k|+|e^j||e^k|}=(-1)^{|e_i||e_j|+|e_k||e_k|}.
\end{equation*}
If $j=k$, $i\neq j$, $|e_i|+|e_j|+|e_k|={\bar 0}$ and $e_i\wedge e_j\wedge e_k\neq 0$, then $|e_i|={\bar 0}$ and $|e_j|={\bar 1}$, thus
\begin{equation*}
(e^i\wedge e^j\wedge e^k,e_i\wedge e_j\wedge e_k)=(e^j\wedge e^i\wedge e^k,e_j\wedge e_i\wedge e_k)=-2=2(-1)^{|e_i||e_j|+|e_k||e_k|}
\end{equation*}
and
\begin{equation*}
(e^j\wedge e^k\wedge e^i,e_j\wedge e_k\wedge e_i)=-2=2(-1)^{|e_j||e_k|+|e_i||e_i|}.
\end{equation*}
If $i=j=k$ and $|e_i|+|e_j|+|e_k|={\bar 0}$, then $$|e_i|={\bar 0}, \quad e_i\wedge e_j\wedge e_k=0.$$
By the above discussion, we have
\begin{equation}\label{eq 3.1+1}
\phi=\frac{1}{6}\sum_{1\leq i, j, k\leq n}(-1)^{|e_i||e_j|+|e_k||e_k|}(\phi,e_i\wedge e_j\wedge e_k)e^i\wedge e^j\wedge e^k.
\end{equation}
Thus,
\begin{equation*}
\begin{aligned}
 \phi^2&=(\phi^2)_0=-(\phi,\phi)\\
       &=-\frac{1}{6}\sum_{1\leq i, j, k\leq n}(-1)^{|e_i||e_j|+|e_k||e_k|}(\phi,e_i\wedge e_j\wedge e_k)(e^i\wedge e^j\wedge e^k,\phi)\\
       &=-\frac{1}{24}\sum_{1\leq i, j, k\leq n}(-1)^{|e_i||e_j|+|e_k||e_k|}([e_i,e_j]^\phi,e_k)([e^i,e^j]^\phi,e^k)\\
       &=-\frac{1}{24}\sum_{1\leq i, j\leq n}(-1)^{|e_i||e_j|}([e_i,e_j]^\phi,[e^i,e^j]^\phi)\\
       &=\frac{1}{24}\sum_{1\leq i, j\leq n}(-1)^{|e_j||e_j|}([e_i,[e^i,e^j]^\phi]^\phi,e_j),
\end{aligned}
\end{equation*}
which implies that
\begin{equation}\label{eq 3.2}
(\phi^2)=(\phi^2)_0=\frac{1}{24}\str\sum_{i=1}^n\ad e_i\ad e^i.
\end{equation}
Since the map $\ad :\mathfrak{g}\rightarrow\End(\mathfrak{g})$ naturally extends to a homomorphism of associative algebras $\ad : T(\mathfrak{g})\rightarrow\End(\mathfrak{g})$, we have that
\begin{equation*}
(\phi^2)=(\phi^2)_0=\frac{1}{24}\str\ad (\Cas_\mathfrak{g}).
\end{equation*}
\begin{theorem}\label{theorem III.3}
For any $\phi\in \mathcal{V}$, let $\mathfrak{g}$ be the corresponding quadratic Lie superalgebra. Then the constant
\begin{equation*}
\phi^2=\frac{1}{24}\str\ad(\Cas_\mathfrak{g}).
\end{equation*}
\end{theorem}

\section{The criterion for $(\nu,(\cdot,\cdot)_\mathfrak{g})$ to be of Lie super type}
Let $\mathfrak{r}$ be a finite dimensional complex Lie superalgebra with a non-degenerate invariant super-symmetric bilinear form $(\cdot,\cdot)_\mathfrak{r}$, let $\mathfrak{p}$ be a finite dimensional complex super vector space with a non-degenerate super-symmetric
bilinear form $(\cdot,\cdot)_\mathfrak{p}$, and let
\begin{equation*}
\nu :\mathfrak{r}\rightarrow\mathfrak{osp}(\mathfrak{p})
\end{equation*}
be a $(\cdot,\cdot)_\mathfrak{p}$-invariant representation of $\mathfrak{r}$ on $\mathfrak{p}$. Define
\begin{equation*}
\mathfrak{g}=\mathfrak{r}\oplus\mathfrak{p},
\end{equation*}
and define a non-degenerate super-symmetric bilinear form $(\cdot,\cdot)_\mathfrak{g}$ on $\mathfrak{g}$ by
$$(\cdot,\cdot)_\mathfrak{g}|_\mathfrak{r}=(\cdot,\cdot)_\mathfrak{r},\quad  (\cdot,\cdot)_\mathfrak{g}|_\mathfrak{p}=(\cdot,\cdot)_\mathfrak{p}, \quad (\mathfrak{p},\mathfrak{r})_\mathfrak{g}=0.$$
The goal of this section is to give a necessary and sufficient condition for $(\nu,(\cdot,\cdot)_\mathfrak{g})$ to be of Lie super type. The case when $\mathfrak{r}=0$ has been studied in Section~\ref{section3}, this section is to study the case for $\mathfrak{r}\not=0$.

Let $r=\dim\mathfrak{r}$, $p=\dim\mathfrak{p}$ and $n=r+p$. Take a homogeneous basis $\{e_1,\ldots,e_n\}$ of $\mathfrak{g}$, given by a basis $\{x_1,\ldots,x_r\}$ of $\mathfrak{r}$ followed by a basis $\{y_1,\ldots,y_p\}$ of $\mathfrak{p}$. Since $\mathfrak{r}$ is $(\cdot,\cdot)_\mathfrak{g}$-orthogonal to $\mathfrak{p}$, we have the $(\cdot,\cdot)_\mathfrak{g}$-dual basis $\{e^1,\ldots,e^n\}=\{x^1,\ldots,x^r,y^1,\ldots,y^p\}$, where $\{x^1,\ldots,x^r\}$ is the $(\cdot,\cdot)_\mathfrak{r}$-dual basis of $\{x_1,\ldots,x_r\}$ and $\{y^1,\ldots,y^p\}$ is the $(\cdot,\cdot)_\mathfrak{p}$-dual basis of $\{y_1,\ldots,y_p\}$.

If $(\nu,(\cdot,\cdot)_\mathfrak{g})$ is of Lie super type, by Theorem \ref{theorem III.2}, the quadratic Lie superalgebraic structure on $\mathfrak{g}$ determines a cubic element $\phi\in\Lambda_{\bar 0}^3\mathfrak{g}$. The corresponding bracket $[\cdot,\cdot]^\phi$ is defined by
\begin{equation}\label{eq IV.1}
    [z_1,z_2]^\phi=2\iota(z_1)\iota(z_2)\phi
\end{equation}
for any $z_1,z_2\in\mathfrak{g}$. The condition (b) in the definition of a Lie super type says that
\begin{equation}\label{eq IV.2}
    [z_1,z_2]^\phi=[z_1,z_2]
\end{equation}
for any $z_1,z_2\in\mathfrak{r}$, and
\begin{equation}\label{eq IV.3}
    [x,y]^\phi=\nu(x)y
\end{equation}
for any $x\in\mathfrak{r}$, $y\in\mathfrak{p}$.

Let $\phi_\mathfrak{r}\in\Lambda_{\bar 0}^3(\mathfrak{r})$ be the element corresponding to the quadratic Lie superalgebraic structure on $\mathfrak{r}$.  Then
\begin{equation*}
    [x_i,x_j]=2\iota(x_i)\iota(x_j)\phi_\mathfrak{r}.
\end{equation*}
By the identity (\ref{eq 3.1+1}) and Lemma \ref{lemma 2.5}, we have
\begin{equation}\label{eq IV.4}
\phi_\mathfrak{r}=-\frac{1}{12}\sum_{1\leq i, j, k\leq r}(-1)^{|x_i||x_j|+|x_k||x_k|}([x_i,x_j],x_k)x^i\wedge x^j\wedge x^k.
\end{equation}
Define a cubic element $\phi_\mathfrak{p}\in\Lambda_{\bar 0}^3(\mathfrak{p})$ by
\begin{equation}\label{eq IV.5}
\phi_\mathfrak{p}=-\frac{1}{12}\sum_{1\leq i, j, k\leq p}(-1)^{|y_i||y_j|+|y_k||y_k|}([y_i,y_j]^\phi,y_k)y^i\wedge y^j\wedge y^k.
\end{equation}

By Lemma \ref{lemma 2.6}, there exists a unique Lie superalgebra homomorphism
\begin{equation*}
\nu_* :\mathfrak{r}\rightarrow\Lambda^2(\mathfrak{p})
\end{equation*}
such that, for any $x\in\mathfrak{r}$ and $y\in\mathfrak{p}$,
\begin{equation}\label{eq IV.5+1}
[\nu_*(x),y]_C=\nu_*(x) y-(-1)^{|\nu_*(x)||y|}y\nu_*(x)=\nu(x)y=[x,y]^\phi.
\end{equation}
Note that $\nu_*$ preserves the grading, that is, $|\nu_*(x)|=|x|$ for any $x\in\mathfrak{r}$. Clearly,
\begin{equation*}
\{y_i\wedge y_j|1\leq i\leq j\leq p,y_i\wedge y_j\neq 0\}
\end{equation*}
is a basis of $\Lambda^2\mathfrak{p}$ with the dual basis
\begin{equation*}
\left\{\frac{1}{(y^i\wedge y^j,y_i\wedge y_j)}y^i\wedge y^j|1\leq i\leq j\leq p,y^i\wedge y^j\neq 0\right\}.
\end{equation*}
Note that
\begin{equation*}
(y^i\wedge y^j,y_i\wedge y_j)=(-1)^{|y^i||y^j|}
\end{equation*}
if $i\neq j$, and
\begin{equation*}
(y^i\wedge y^i,y_i\wedge y_i)=-2=2(-1)^{|y^i||y^i|}
\end{equation*}
for any $y_i\in\mathfrak{p}_{\bar 1}$. Then for any $u\in\Lambda^2\mathfrak{p}$,
\begin{equation}\label{eq IV.5+2}
u=\sum_{1\leq i\leq j\leq p,\atop y_i\wedge y_j\neq 0}\frac{(u,y_i\wedge y_j)}{(y^i\wedge y^j,y_i\wedge y_j)}y^i\wedge y^j
    =\frac{1}{2}\sum_{1\leq i, j\leq p}(-1)^{|y^i||y^j|}(u,y_i\wedge y_j)y^i\wedge y^j.
\end{equation}
It follows that
\begin{equation*}
\begin{aligned}
\nu_*(x)=&\frac{1}{2}\sum_{1\leq i, j\leq p}(-1)^{|y^i||y^j|}(\nu_*(x),y_i\wedge y_j)y^i\wedge y^j\\
    =&\frac{1}{2}\sum_{1\leq i, j\leq p}(-1)^{|y^i||y^j|+|\nu_*(x)||y_i|}(\iota(y_i)\nu_*(x),y_j)y^i\wedge y^j\\
    =&-\frac{1}{4}\sum_{1\leq i, j\leq p}(-1)^{|y^i||y^j|}([x,y_i]^\phi,y_j)y^i\wedge y^j,
\end{aligned}
\end{equation*}
since
\begin{equation*}
 [x,y_i]^\phi=[\nu_*(x),y]_C=-2(-1)^{|\nu_*(x)||y_i|}\iota(y_i)\nu_*(x).
\end{equation*}
Let
\begin{equation}\label{eq IV.6}
\phi_\nu=\sum_{1\leq i\leq r}\nu_*(x_i)\wedge x^i.
\end{equation}
By a direct calculation, we have
\begin{equation*}
\begin{aligned}
\phi =&\sum_{1\leq i\leq j\leq k\leq n,\atop e_i\wedge e_j\wedge e_k\neq 0}\frac{(\phi,e_i\wedge e_j\wedge e_k)}{(e^i\wedge e^j\wedge e^k,e_i\wedge e_j\wedge e_k)}e^i\wedge e^j\wedge e^k\\
    =&-\frac{1}{2}\sum_{1\leq i\leq j\leq k\leq n,\atop e_i\wedge e_j\wedge e_k\neq 0}\frac{(([e_i,e_j]^\phi,e_k)}{(e^i\wedge e^j\wedge e^k,e_i\wedge e_j\wedge e_k)}e^i\wedge e^j\wedge e^k\\
    =&-\frac{1}{12}\sum_{1\leq i, j, k\leq r}(-1)^{|x_i||x_j|+|x_k||x_k|}([x_i,x_j],x_k)x^i\wedge x^j\wedge x^k\\
       &-\frac{1}{4}\sum_{1\leq i\leq r,1\leq j, k\leq p}(-1)^{|x_i||x_i|+|y_j||y_k|}([x_i,y_j]^\phi,y_k)x^i\wedge y^j\wedge y^k\\
     &-\frac{1}{12}\sum_{1\leq i, j, k\leq p}(-1)^{|y_i||y_j|+|y_k||y_k|}([y_i,y_j]^\phi,y_k)y^i\wedge y^j\wedge y^k\\
    =&\phi_\mathfrak{r}+\sum_{1\leq i\leq r}(-1)^{|x_i||x_i|}x^i\wedge\nu_*(x_i)+\phi_\mathfrak{p}\\
    =&\phi_\mathfrak{r}+\phi_\nu+\phi_\mathfrak{p}.
\end{aligned}
\end{equation*}
Furthermore, by the identity (\ref{eq 2.2.4-2}), for any homogeneous element $x\in\mathfrak{r}$,
\begin{equation*}
\begin{aligned}
 &([\nu_*(x),\phi_\mathfrak{p}]_C,y_i\wedge y_j\wedge y_k)\\
=&-(\phi_\mathfrak{p},[\nu_*(x),y_i\wedge y_j\wedge y_k]_C)\\
=&-(\phi_\mathfrak{p},[\nu_*(x),y_i]_C\wedge y_j\wedge y_k)-(-1)^{|x||y_i|}(\phi_\mathfrak{p},y_i\wedge [\nu_*(x),y_j]_C\wedge y_k)\\
 &-(-1)^{|x|(|y_i|+|y_j|)}(\phi_\mathfrak{p},y_i\wedge y_j\wedge [\nu_*(x),y_k]_C)\\
=&-(\phi,[x,y_i]^\phi\wedge y_j\wedge y_k)-(-1)^{|x||y_i|}(\phi,y_i\wedge [x,y_j]^\phi\wedge y_k)\\
 &-(-1)^{|x|(|y_i|+|y_j|)}(\phi,y_i\wedge y_j\wedge [x,y_k]^\phi)\\
=&-([[x,y_i]^\phi,y_j]^\phi,y_k)-(-1)^{|x||y_i|}([y_i,[x,y_j]^\phi]^\phi,y_k)-(-1)^{|x|(|y_i|+|y_j|)}([y_i,y_j]^\phi,[x,y_k]^\phi)\\
=&-([[x,y_i]^\phi,y_j]^\phi,y_k)-(-1)^{|x||y_i|}([y_i,[x,y_j]^\phi]^\phi,y_k)+([x,[y_i,y_j]^\phi]^\phi,y_k).
\end{aligned}
\end{equation*}
It follows that $[\nu_*(x),\phi_\mathfrak{p}]_C=0$ by the super Jacobi identity, that is, $$\phi_\mathfrak{p}\in (\Lambda_{\bar 0}^3(\mathfrak{p}))^\mathfrak{r}.$$

\begin{remark}
If $(\nu,(\cdot,\cdot)_\mathfrak{g})$ is of Lie super type, $\phi_\mathfrak{r}$ and $\phi_\nu$ are determined completely by the quadratic Lie superalgebraic structure of $\mathfrak{r}$ and the $\mathfrak{r}$-module structure on $\mathfrak{p}$. In order to give a Lie superalgebraic structure satisfying the conditions of a Lie super type, we only need to determine the cubic element $\phi_\mathfrak{p}$.
\end{remark}

Conversely, for any $\phi_\mathfrak{p}\in(\Lambda_{\bar 0}^3(\mathfrak{p}))^\mathfrak{r}$, define the cubic element $\phi\in\Lambda_{\bar 0}^3\mathfrak{g}$ by
\begin{equation*}
\phi=\phi_\mathfrak{r}+\phi_\nu+\phi_\mathfrak{p},
\end{equation*}
where $\phi_\mathfrak{r}$ and $\phi_\nu$ are defined by identities (\ref{eq IV.4}) and (\ref{eq IV.6}), respectively.
Let $[\cdot,\cdot]^\phi$ be the bracket on $\mathfrak{g}$ determined by the identity (\ref{eq IV.1}). By Section~\ref{section3}, we know that $[\cdot,\cdot]^\phi$ has skew super-symmetry and $(\cdot,\cdot)_\mathfrak{g}$ is invariant with respect to the bracket $[\cdot,\cdot]^\phi$.
Moreover,
\begin{equation*}
[x_i,x_j]^\phi=2\iota(x_i)\iota(x_j)\phi=2\iota(x_i)\iota(x_j)\phi_\mathfrak{r}=[x_i,x_j]
\end{equation*}
for any $1\leq i,j\leq r$ and
\begin{equation*}
[x_i,y_j]^\phi=2\iota(x_i)\iota(y_j)\phi=2\iota(x_i)\iota(y_j)\phi_\mathfrak{\nu}=2(-1)^{|x_i||y_j|}\iota(y_j)\nu_*(x_i)=[\nu_*(x_i),y_j]_C=\nu(x_i)y_j
\end{equation*}
for any $1\leq i\leq r$ and $1\leq j\leq p$.

Clearly, the super Jacobi identity holds for $x_i,x_j,x_k$ since $\mathfrak{r}$ is a Lie superalgebra and for $x_i,x_j,y_k$ since $\mathfrak{p}$ is a $\mathfrak{r}$-module.

We claim that the super Jacobi identity holds for $x_i,y_j,y_k$. In fact, let $P_\mathfrak{r}:\mathfrak{g}\rightarrow\mathfrak{r}$ and $P_\mathfrak{p}:\mathfrak{g}\rightarrow\mathfrak{p}$ be the projections with respect to the decomposition $\mathfrak{g}=\mathfrak{r}\oplus\mathfrak{p}$.
Since $[\cdot,\cdot]^\phi$ has skew super-symmetry and $(\cdot,\cdot)_\mathfrak{g}$ is invariant with respect to the bracket $[\cdot,\cdot]^\phi$, we have
\begin{equation*}
\begin{aligned}
 &([x_i,[y_j,y_k]^\phi]^\phi-[[x_i,y_j]^\phi,y_k]^\phi-(-1)^{|x_i||y_j|}[y_j,[x_i,y_k]^\phi]^\phi,x_l)\\
=&-(-1)^{|x_i|(|y_j|+|y_k|)}([y_j,y_k]^\phi,[x_i,x_l]^\phi)-([x_i,y_j]^\phi,[y_k,x_l]^\phi)-(-1)^{|x_i||y_j|}(y_j,[[x_i,y_k]^\phi,x_l]^\phi)\\
=&-(-1)^{|x_i|(|y_j|+|y_k|)}(y_j,[y_k,[x_i,x_l]^\phi]^\phi-(-1)^{|x_i||y_k|}[x_i,[y_k,x_l]^\phi]^\phi-[[y_k,x_i]^\phi,x_l]^\phi)\\
=&0
\end{aligned}
\end{equation*}
for any $1\leq l\leq r$. It follows that
\begin{equation}\label{eq IV.7}
P_\mathfrak{r}([x_i,[y_j,y_k]^\phi]^\phi-[[x_i,y_j]^\phi,y_k]^\phi-(-1)^{|x_i||y_j|}[y_j,[x_i,y_k]^\phi]^\phi)=0.
\end{equation}
Since $\phi_\mathfrak{p}\in(\Lambda_{\bar 0}^3(\mathfrak{p}))^\mathfrak{r}$, for any $1\leq l\leq p$, we have
\begin{equation*}
\begin{aligned}
&([x_i,[y_j,y_k]^\phi]^\phi-[[x_i,y_j]^\phi,y_k]^\phi-(-1)^{|x_i||y_j|}[y_j,[x_i,y_k]^\phi]^\phi,y_l)\\
=&([\nu_*(x),\phi_\mathfrak{p}]_C,y_j\wedge y_k\wedge y_l)=0,
\end{aligned}
\end{equation*}
that is,
\begin{equation}\label{eq IV.8}
P_\mathfrak{p}([x_i,[y_j,y_k]^\phi]^\phi-[[x_i,y_j]^\phi,y_k]^\phi-(-1)^{|x_i||y_j|}[y_j,[x_i,y_k]^\phi]^\phi)=0.
\end{equation}
Then the claim holds by identities (\ref{eq IV.7}) and (\ref{eq IV.8}).

At last, we consider the super Jacobi identity for $y_i,y_j,y_k$. For any $1\leq l\leq r$, we have
\begin{equation*}
\begin{aligned}
 &([y_i,[y_j,y_k]^\phi]^\phi-[[y_i,y_j]^\phi,y_k]^\phi-(-1)^{|y_i||y_j|}[y_j,[y_i,y_k]^\phi]^\phi,x_l)\\
=&(y_i,[[y_j,y_k]^\phi,x_l]^\phi)-([y_i,y_j]^\phi,[y_k,x_l]^\phi)+(-1)^{|y_j||y_k|}([y_i,y_k]^\phi,[y_j,x_l]^\phi)\\
=&(y_i,[[y_j,y_k]^\phi,x_l]^\phi-[y_j,[y_k,x_l]^\phi]^\phi+(-1)^{|y_j||y_k|}[y_k,[y_j,x_l]^\phi]^\phi)\\
=&0.
\end{aligned}
\end{equation*}
It follows that
\begin{equation}\label{eq 4.9}
P_\mathfrak{r}([y_i,[y_j,y_k]^\phi]^\phi-[[y_i,y_j]^\phi,y_k]^\phi-(-1)^{|y_i||y_j|}[y_j,[y_i,y_k]^\phi]^\phi)=0.
\end{equation}
Thus we only need to consider
\begin{equation*}
P_\mathfrak{p}([y_i,[y_j,y_k]^\phi]^\phi-[[y_i,y_j]^\phi,y_k]^\phi-(-1)^{|y_i||y_j|}[y_j,[y_i,y_k]^\phi]^\phi).
\end{equation*}
By Lemma \ref{lemma III.1}, we have
\begin{equation}\label{eq 4.10}
[y_i,[y_j,y_k]^\phi]^\phi-[[y_i,y_j]^\phi,y_k]^\phi-(-1)^{|y_i||y_j|}[y_j,[y_i,y_k]^\phi]^\phi=2\iota(y_i)\iota(y_j)\iota(y_k)\phi^2.
\end{equation}
Note that, for any $1\leq l\leq p$,
\begin{equation}\label{eq 4.11}
(\iota(y_i)\iota(y_j)\iota(y_k)\phi^2,y_l)=(\iota(y_i)\iota(y_j)\iota(y_k)(\phi_\mathfrak{\nu}^2+\phi_\mathfrak{p}^2),y_l).
\end{equation}
By Lemma \ref{lemma 2.1}, $\phi_\nu$ is independent of the choice of basis of $\mathfrak{r}$. It follows that
\begin{equation*}
\phi_\nu=\sum_{1\leq i\leq r}\nu_*(x_i)\wedge x^i=\sum_{1\leq j\leq r}(-1)^{|x_j||x^j|}\nu_*(x^j)\wedge x_j.
\end{equation*}
Thus,
\begin{equation*}
\phi_\mathfrak{\nu}^2=(\sum_{1\leq i\leq r}\nu_*(x_i)\wedge x^i)(\sum_{1\leq j\leq r}(-1)^{|x_j||x^j|}\nu_*(x^j)\wedge x_j).
\end{equation*}
Since $\nu_*(x_i)\in\Lambda^2(\mathfrak{p})$, we have $\iota(x^i)\nu_*(x_i)=0$. By Lemma \ref{lemma 2.10},
\begin{equation*}
\begin{aligned}
\nu_*(x_i)x^i&=\alpha^2(\nu_*(x_i)x^i)=(-1)^{|x_i||x^i|}\alpha(\alpha(x^i)\alpha(\nu_*(x_i)))=-(-1)^{|x_i||x^i|}\alpha(x^i\nu_*(x_i))\\
             &=-(-1)^{|x_i||x^i|}\alpha(x^i\wedge\nu_*(x_i)+\iota(x^i)\nu_*(x_i))=-(-1)^{|x_i||x^i|}\alpha(x^i\wedge\nu_*(x_i))\\
             &=-\alpha(\nu_*(x_i))\wedge\alpha(x^i)=\nu_*(x_i)\wedge x^i,
\end{aligned}
\end{equation*}
which implies that
\begin{equation}\label{eq 4.11+1}
\begin{aligned}
\phi_\mathfrak{\nu}^2&=(\sum_{1\leq s\leq r}\nu_*(x_s)x^s)(\sum_{1\leq t\leq r}(-1)^{|x_t||x^t|}\nu_*(x^t)\wedge x_t)\\
                     &=(\sum_{1\leq s\leq r}\nu_*(x_s))(\sum_{1\leq t\leq r}(-1)^{|x_t||x^t|}x^s\wedge\nu_*(x^t)\wedge x_t)+\sum_{1\leq s\leq r}\nu_*(x_s)\nu_*(x^s).
\end{aligned}
\end{equation}
Note that
\begin{equation*}
(\iota(y_i)\iota(y_j)\iota(y_k)((\sum_{1\leq s\leq r}\nu_*(x_s))(\sum_{1\leq t\leq r}(-1)^{|x_t||x^t|}x^s\wedge\nu_*(x^t)\wedge x_t),y_l)=0.
\end{equation*}
It follows that
\begin{equation}\label{eq 4.12}
 (\iota(y_i)\iota(y_j)\iota(y_k)\phi_\mathfrak{\nu}^2,y_l)
=(\iota(y_i)\iota(y_j)\iota(y_k)(\sum_{1\leq s\leq r}\nu_*(x_s) \nu_*(x^s)),y_l).
\end{equation}

Let $\Lambda^\text{even}(\mathfrak{p})=\sum\limits_{i=0}\Lambda^{2i}(\mathfrak{p})$. Then $\Lambda^\text{even}(\mathfrak{p})$ is a subalgebra of $\Lambda(\mathfrak{p})$ with respect to the Clifford multiplication. We can extent $\nu_*:\mathfrak{r}\rightarrow\Lambda^2(\mathfrak{p})$ to a homomorphism of associative algebras
\begin{equation*}
\nu_*:T(\mathfrak{r})\rightarrow\Lambda^\text{even}(\mathfrak{p}).
\end{equation*}
By identities (\ref{eq 4.11}) and (\ref{eq 4.12}), we have
\begin{equation}\label{eq 4.13}
(\iota(y_i)\iota(y_j)\iota(y_k)\phi^2,y_l)=(\iota(y_i)\iota(y_j)\iota(y_k)(\nu_*(\Cas_{\mathfrak r})+\phi_\mathfrak{p}^2),y_l)
\end{equation}
It is clear that $(\nu_*(\Cas_{\mathfrak r}))_k=0$ if $k\notin\{0,2,4\}$. By Lemma \ref{lemma 2.1},
\begin{equation*}
\nu_*(\Cas_\mathfrak{r})=\sum\limits_{i=1}^r\nu_*(x_i)\nu_*(x^i)=\sum_{i=1}^r(-1)^{|x_i||x^i|}\nu_*(x^i)\nu_*(x_i).
\end{equation*}
Then, by Lemma \ref{lemma 2.10},
\begin{equation*}
\alpha(\nu_*(\Cas_\mathfrak{r}))=\sum\limits_{i=1}^r\alpha(\nu_*(x_i)\nu_*(x^i))=\sum_{i=1}^r(-1)^{|x_i||x^i|}\nu_*(x^i)\nu_*(x_i)=\nu_*(\Cas_\mathfrak{r}),
\end{equation*}
which implies that
\begin{equation*}
\nu_*(\Cas_\mathfrak{r})=(\nu_*(\Cas_\mathfrak{r}))_4+(\nu_*(\Cas_\mathfrak{r}))_0.
\end{equation*}
Clearly, $(\phi_\mathfrak{p}^2)_k=0$ if $k\notin\{0,2,4,6\}$. Note that
\begin{equation*}
\alpha(\phi_\mathfrak{p}^2)=(-1)^{|\phi_\mathfrak{p}||\phi_\mathfrak{p}|}\alpha(\phi_\mathfrak{p})\alpha(\phi_\mathfrak{p})=\phi_\mathfrak{p}^2,
\end{equation*}
which implies that
\begin{equation*}
\phi_\mathfrak{p}^2=(\phi_\mathfrak{p}^2)_4+(\phi_\mathfrak{p}^2)_0.
\end{equation*}
Hence
\begin{equation*}
\nu_*(\Cas_\mathfrak{r})+\phi_\mathfrak{p}^2=(\nu_*(\Cas_\mathfrak{r})+\phi_\mathfrak{p}^2)_4+(\nu_*(\Cas_\mathfrak{r})+\phi_\mathfrak{p}^2)_0.
\end{equation*}
Furthermore, by identities (\ref{eq 4.10}) and (\ref{eq 4.13}), $$P_\mathfrak{p}([y_i,[y_j,y_k]^\phi]^\phi-[[y_i,y_j]^\phi,y_k]^\phi-(-1)^{|y_i||y_j|}[y_j,[y_i,y_k]^\phi]^\phi)=0$$ if and only if $$(\nu_*(\Cas_\mathfrak{r})+\phi_\mathfrak{p}^2)_4=0,$$ that is, $\nu_*(\Cas_\mathfrak{r})+\phi_\mathfrak{p}^2$ is a constant.

In summary, we have the following theorem.
\begin{theorem}\label{theorem IV.1}
There is a one-to-one correspondence between the set $$\mathcal{V}=\{\phi_\mathfrak{p}\in(\Lambda_{\bar 0}^3(\mathfrak{p}))^\mathfrak{r}|\nu_*(\Cas_\mathfrak{r})+\phi_\mathfrak{p}^2\in\mathbb{C}\}$$ and the set of the pairs $(\nu,(\cdot,\cdot)_\mathfrak{g})$ of Lie super type. For any $\phi_\mathfrak{p}\in\mathcal{V}$, set $\phi=\phi_\mathfrak{r}+\phi_\nu+\phi_\mathfrak{p}$, the bracket of the Lie superalgebra corresponding to $\phi_\mathfrak{p}$ is defined by
\begin{equation*}
[z_1,z_2]^{\phi_\mathfrak{p}}=2\iota(z_1)\iota(z_2)\phi, \quad \forall z_1,z_2\in\mathfrak{g}.
\end{equation*}
\end{theorem}

For any $\phi_\mathfrak{p}\in\mathcal{V}$, set $\phi=\phi_\mathfrak{r}+\phi_\nu+\phi_\mathfrak{p}$. Let $\mathfrak{g}$ be the corresponding quadratic Lie superalgebra with the bracket $[z_1,z_2]^{\phi_\mathfrak{p}}=2\iota(z_1)\iota(z_2)\phi$. We denote by $\ad_\mathfrak{g}$ (resp. $\ad_\mathfrak{r}$) the adjoint representation of $\mathfrak{g}$ on itself (resp. $\mathfrak{r}$ on itself) and that extended to $U(\mathfrak{g})$ (resp. $U(\mathfrak{r})$).

Recall that $\phi_\mathfrak{r}\in\Lambda^3(\mathfrak{r})$, $\phi_\nu\in\pi_\Lambda(T^2(\mathfrak{p})\otimes T(\mathfrak{r}))$ and $\phi_\mathfrak{p}\in\Lambda^3(\mathfrak{p})$. Since $\mathfrak{r}$ is $(\cdot,\cdot)_\mathfrak{g}$-orthogonal to $\mathfrak{p}$, we have
\begin{equation*}
(\phi^2)_0=(\phi_\mathfrak{r}^2)_0+(\phi_\nu^2)_0+(\phi_\mathfrak{p}^2)_0
\end{equation*}
By the identity (\ref{eq 4.11+1}), we have
\begin{equation*}
(\phi_\nu^2)_0=(\sum_{1\leq i\leq r}\nu_*(x_i)\nu_*(x^i))_0=(\nu_*(\Cas_\mathfrak{r}))_0.
\end{equation*}
Then by Theorem \ref{theorem III.2}, \ref{theorem III.3} and \ref{theorem IV.1}, we have
\begin{equation*}
\begin{aligned}
\nu_*(\Cas_\mathfrak{r})+\phi_\mathfrak{p}^2&=(\nu_*(\Cas_\mathfrak{r})+\phi_\mathfrak{p}^2)_0=(\phi_\nu^2)_0+(\phi_\mathfrak{p}^2)_0\\
                                            &=(\phi^2)_0-(\phi_\mathfrak{r}^2)_0=\frac{1}{24}(\str\ad_{\mathfrak g}(\Cas_{\mathfrak g})-\str\ad_{\mathfrak r}(\Cas_{\mathfrak r})).
\end{aligned}
\end{equation*}

\begin{theorem}\label{theorem IV.2}
Assume that $(\nu,(\cdot,\cdot)_{\mathfrak g})$ is of Lie super type corresponding to $\phi_\mathfrak{p}\in\mathcal{V}$, i.e. the bracket of $\mathfrak g$ is defined by $[z_1,z_2]^{\phi_\mathfrak{p}}=2\iota(z_1)\iota(z_2)\phi$, where $\phi=\phi_\mathfrak{r}+\phi_\nu+\phi_\mathfrak{p}$. Then the constant
\begin{equation}\label{scalar}
\nu_*(\Cas_{\mathfrak r})+\phi_\mathfrak{p}^2=\frac{1}{24}(\str\ad_{\mathfrak g}(\Cas_{\mathfrak g})-\str\ad_{\mathfrak r}(\Cas_{\mathfrak r})).
\end{equation}
\end{theorem}

\section{Dirac operator for quadratic Lie superalgebras}
Let ${\mathfrak g}$ be a finite dimensional complex quadratic Lie superalgebra with respect to $(\cdot,\cdot)$, let ${\mathfrak r}$ be a subalgebra of ${\mathfrak g}$ such that the restriction of $(\cdot,\cdot)$ to ${\mathfrak r}$ is non-degenerate, let ${\mathfrak g}={\mathfrak r}\oplus{\mathfrak p}$ be the orthogonal decomposition with respect to $(\cdot,\cdot)$, and let $\nu$ be the adjoint representation of ${\mathfrak r}$ on ${\mathfrak p}$. Then $(\nu, (\cdot,\cdot))$ is of Lie super type.

Analogous to the cubic Dirac operator of quadratic Lie algebra introduced by Kostant in \cite{Ko2}, we will define the cubic Dirac operator $D({\mathfrak g},{\mathfrak r})$ of $\mathfrak g$ corresponding to the above decomposition ${\mathfrak g}={\mathfrak r}\oplus{\mathfrak p}$. Denote by $\xi$ the injection map $\mathfrak{g}\rightarrow U(\mathfrak{g})$ and its extension $U(\mathfrak{g})\rightarrow U(\mathfrak{g})$. Define the cubic Dirac operator $D({\mathfrak g},{\mathfrak r})\in U({\mathfrak g})\otimes C({\mathfrak p})$ by
\begin{equation}\label{eq 5.1}
D({\mathfrak g},{\mathfrak r})=\sum_{i=1}^p \xi(y_i)\otimes y^i+1\otimes\phi_{\mathfrak p},
\end{equation}
where
$
\phi_\mathfrak{p}=-\frac{1}{12}\sum_{1\leq i, j, k\leq p}(-1)^{|y_i||y_j|+|y_k||y_k|}([y_i,y_j],y_k)y^i\wedge y^j\wedge y^k.
$
It is clear that $D({\mathfrak g},{\mathfrak r})$ is independent of the choice of basis. Note that
\begin{equation}\label{eq 5.2}
(\phi_{\mathfrak p},y_i\wedge y_j\wedge y_k)=(\phi,y_i\wedge y_j\wedge y_k)=-\frac{1}{2}([y_i,y_j],y_k)
\end{equation}
and $|y_i|=|y^i|$ for $1\leq i \leq p$.
Set $\Box_1=\sum_{i=1}^p \xi(y_i)\otimes y^i$ and $\Box_2=1\otimes\phi_{\mathfrak p}$. Then
\begin{equation*}
D({\mathfrak g},{\mathfrak r})=\Box_1+\Box_2.
\end{equation*}
Recall that  $U({\mathfrak g})$ and $C({\mathfrak p})$ are both superalgebras and the multiplication on $U({\mathfrak g})\otimes C({\mathfrak p})$ is defined by the identity (\ref{eq 2.1.2}). Since $y^j y^i+(-1)^{|y_j||y_i|}y^i y^j=2(y^j,y^i)$, $y^i y^j=y^i\wedge y^j+(y^i,y^j)$, and $\xi(y_i)\xi(y_j)-(-1)^{|y_i||y_j|}\xi(y_j)\xi(y_i)=\xi([y_i,y_j])$ we have
\begin{equation*}
\begin{aligned}
 (\Box_1)^2=&(\sum_{i=1}^p\xi(y_i)\otimes y^i)(\sum_{j=1}^p\xi(y_j)\otimes y^j)=\sum_{1\leq i,j\leq p}(-1)^{|y_i||y_j|}\xi(y_i)\xi(y_j)\otimes y^i y^j\\
=&\frac{1}{2}\sum_{1\leq i,j\leq p}((-1)^{|y_i||y_j|}\xi(y_i)\xi(y_j)\otimes y^i y^j+(-1)^{|y_j||y_i|}\xi(y_j)\xi(y_i)\otimes y^j y^i)\\
=&\sum_{1\leq i,j\leq p}(-1)^{|y_j||y_i|}(y^j,y^i)\xi(y_j)\xi(y_i)\otimes 1+\frac{1}{2}\sum_{1\leq i,j\leq p}(-1)^{|y_i||y_j|}\xi([y_i,y_j])\otimes y^i y^j\\
=&\sum_{1\leq i,j\leq p}(y^i,y^j)\xi(y_j)\xi(y_i)\otimes 1+\frac{1}{2}\sum_{1\leq i,j\leq p}(-1)^{|y_i||y_j|}\xi([y_i,y_j])\otimes y^i\wedge y^j\\
&+\frac{1}{2}\sum_{1\leq i,j\leq p}(-1)^{|y_i||y_j|}(y^i,y^j)\xi([y_i,y_j])\otimes 1.
\end{aligned}
\end{equation*}
Since $y^j=\sum_{k=1}^p(y^k,y^j)y_k$, we have
\begin{equation*}
\sum_{1\leq i,j\leq p}(y^i,y^j)\xi(y_j)\xi(y_i)\otimes 1=\sum_{1\leq i,j,k\leq p}(y^k,y^j)(y^i,y_k)\xi(y_j)\xi(y_i)\otimes 1.
\end{equation*}
Note that $|y_j|=|y_k|$ if $(y^k,y^j)\neq 0$. Then $\sum_{j=1}^p(y^k,y^j)\xi(y_j)=(-1)^{|y_k||y_k|}\xi(y^k)$.
Hence
\begin{equation*}
\begin{aligned}
\sum_{1\leq i,j\leq p}(y^i,y^j)\xi(y_j)\xi(y_i)\otimes 1=&\sum_{1\leq i,k\leq p}(-1)^{|y_k||y_k|}(y^i,y_k)\xi(y^k)\xi(y_i)\otimes 1 \\ &=\sum_{i=1}^p(-1)^{|y_i||y_i|}\xi(y^i)\xi(y_i)\otimes 1.
\end{aligned}
\end{equation*}
By Lemma \ref{lemma 2.1}, we have
\begin{equation*}
\sum_{1\leq i,j\leq p}(y^i,y^j)\xi(y_j)\xi(y_i)\otimes 1=\sum_{i=1}^p\xi(y_i)\xi(y^i)\otimes 1.
\end{equation*}
Since
\begin{equation*}
\begin{aligned}
 &\sum_{1\leq i,j\leq p}(-1)^{|y_i||y_j|}(y^i,y^j)\xi([y_i,y_j])\otimes 1\\
=&\frac{1}{2}\sum_{1\leq i,j\leq p}((-1)^{|y_i||y_j|}(y^i,y^j)\xi([y_i,y_j])\otimes 1+(-1)^{|y_j||y_i|}(y^j,y^i)\xi([y_j,y_i])\otimes 1)\\
=&\frac{1}{2}\sum_{1\leq i,j\leq p}(y^j,y^i)\xi([y_i,y_j]+(-1)^{|y_j||y_i|}[y_j,y_i])\otimes 1\\
=&0,
\end{aligned}
\end{equation*}
we have
\begin{equation*}
\begin{aligned}
(\Box_1)^2=&\sum_{i=1}^p\xi(y_i)\xi(y^i)\otimes 1+\frac{1}{2}\sum_{1\leq i,j\leq p}(-1)^{|y_i||y_j|}\xi([y_i,y_j])\otimes y^i\wedge y^j\\
=&\sum_{i=1}^p\xi(y_i)\xi(y^i)\otimes 1+\frac{1}{2}\sum_{1\leq i,j\leq p}(-1)^{|y_i||y_j|}\xi([y_i,y_j]_{\mathfrak r})\otimes y^i\wedge y^j\\
           &+\frac{1}{2}\sum_{1\leq i,j\leq p}(-1)^{|y_i||y_j|}\xi([y_i,y_j]_{\mathfrak p})\otimes y^i\wedge y_j.
\end{aligned}
\end{equation*}
Denote by $I,II$ and $III$ the three summands in the right side of the above equation respectively, that is,
\begin{equation}\label{eq 5.3}
(\Box_1)^2=I+II+III.
\end{equation}
By identities (\ref{eq 2.2.4}), (\ref{eq 2.2.4-1}), (\ref{eq IV.5+1}) and (\ref{eq IV.5+2})
\begin{equation}\label{eq 5.4}
\begin{aligned}
II=&\frac{1}{2}\sum_{1\leq i,j\leq p}(-1)^{|y_i||y_j|}\xi([y_i,y_j]_{\mathfrak r})\otimes y^i\wedge y^j\\
=&\frac{1}{2}\sum_{1\leq i,j\leq p}\sum_{k=1}^r(-1)^{|y_i||y_j|}(x^k,[y_i,y_j])\xi(x_k)\otimes y^i\wedge y^j\\
=&\frac{1}{2}\sum_{1\leq i,j\leq p}\sum_{k=1}^r(-1)^{|y_i||y_j|}([x^k,y_i],y_j)\xi(x_k)\otimes y^i\wedge y^j\\
=&-\sum_{1\leq i,j\leq p}\sum_{k=1}^r(-1)^{|y_i|(|x^k|+|y_j|)}(\iota(y_i)\nu_*(x^k),y_j)\xi(x_k)\otimes y^i\wedge y^j\\\\
=&-\sum_{k=1}^r\sum_{1\leq i,j\leq p}(-1)^{|y_i||y_j|}(\nu_*(x^k),y_i\wedge y_j)\xi(x_k)\otimes y^i\wedge y^j\\
=&-2\sum_{k=1}^r\xi(x_k)\otimes \nu_*(x^k).
\end{aligned}
\end{equation}
Define a diagonal embedding
$\zeta: \mathfrak r\rightarrow U(\mathfrak g)\otimes C(\mathfrak p)$ by $$\zeta(x)=\xi(x)\otimes 1+1\otimes \nu_*(x),\quad \forall x\in  \mathfrak r.$$ Extending $\zeta$ to a homomorphism $\zeta:\ U({\mathfrak r})\rightarrow U(\mathfrak g)\otimes C(\mathfrak p)$, by Lemma \ref{lemma 2.1}, we have
\begin{equation*}
\begin{aligned}
\zeta(\Cas_{\mathfrak r})=&\sum_{i=1}^r(\xi(x_i)\otimes 1+1\otimes \nu_*(x_i))(\xi(x^i)\otimes 1+1\otimes \nu_*(x^i))\\
                         =&\sum_{i=1}^r(\xi(x_i)\xi(x^i)\otimes 1+\xi(x_i)\otimes\nu_*(x^i)+(-1)^{|x_i||x^i|}\xi(x^i)\otimes\nu_*(x_i)+1\otimes \nu_*(x_i)\nu_*(x^i))\\
                         =&\sum_{i=1}^r(\xi(x_i)\xi(x^i)\otimes 1+2\xi(x_i)\otimes\nu_*(x^i)+1\otimes \nu_*(x_i)\nu_*(x^i)).
\end{aligned}
\end{equation*}
It follows from identities (\ref{eq 5.3}) and (\ref{eq 5.4}) that
\begin{equation}\label{eq 5.5}
I+II+\zeta(\Cas_{\mathfrak r})=\xi(\Cas_{\mathfrak g})\otimes 1+1\otimes \nu_*(\Cas_{\mathfrak r}).
\end{equation}
By Lemma \ref{lemma 2.11} and the identity (\ref{eq IV.5+2}), we have
\begin{equation}\label{eq 5.6}
\begin{aligned}
\Box_1\Box_2+\Box_2\Box_1=&\sum_{k=1}^p\xi(y_k)\otimes(y^k \phi_{\mathfrak p}+\phi_{\mathfrak p} y^k)=2\sum_{k=1}^p \xi(y_k)\otimes \iota(y^k)\phi_{\mathfrak p}\\
=&\sum_{k=1}^p\sum_{1\leq i,j\leq p}(-1)^{|y_i||y_j|}(\iota(y^k)\phi_{\mathfrak p},y_i\wedge y_j)\xi(y_k)\otimes y^i\wedge y^j\\
=&\sum_{1\leq i,j,k\leq p}(-1)^{|y_i||y_j|}(\phi_{\mathfrak p},y^k\wedge y_i\wedge y_j)\xi(y_k)\otimes y^i\wedge y^j\\
=&-\frac{1}{2}\sum_{1\leq i,j,k\leq p}(-1)^{|y_i||y_j|}([y^k,y_i],y_j)\xi(y_k)\otimes y^i\wedge y^j\\
=&-\frac{1}{2}\sum_{1\leq i,j,k\leq p}(-1)^{|y_i||y_j|}(y^k,[y_i,y_j])\xi(y_k)\otimes y^i\wedge y^j\\
=&-\frac{1}{2}\sum_{1\leq i,j\leq p}(-1)^{|y_i||y_j|}\xi([y_i,y_j])\otimes y^i\wedge y^j\\
=&-III.
\end{aligned}
\end{equation}
By identities (\ref{eq 5.3}), (\ref{eq 5.5}) and (\ref{eq 5.6}), we have
\begin{equation*}
(D({\mathfrak g},{\mathfrak r}))^2=\xi(\Cas_{\mathfrak g})\otimes 1-\zeta(\Cas_{\mathfrak r})+1\otimes (\nu_*(\Cas_{\mathfrak r})+\phi_\mathfrak{p}^2).
\end{equation*}
By Theorem \ref{theorem IV.2}, we have an analogue of the Parthasarathy's formula.
\begin{theorem}\label{theorem 5.1}
Let ${\mathfrak g}$ be a finite dimensional complex quadratic Lie superalgebra with respect to $(\cdot,\cdot)$ and let ${\mathfrak r}$ be a subalgebra of ${\mathfrak g}$ such that the restriction of $(\cdot,\cdot)$ to ${\mathfrak r}$ is non-degenerate. Define $D({\mathfrak g},{\mathfrak r})\in U({\mathfrak g})\otimes C({\mathfrak p})$ by the identity (\ref{eq 5.1}). Then
\begin{equation}\label{eq 5.7}
(D({\mathfrak g},{\mathfrak r}))^2=\xi(\Cas_{\mathfrak g})\otimes 1-\zeta(\Cas_{\mathfrak r})+\frac{1}{24}(\str\ad_{\mathfrak g}(\Cas_{\mathfrak g})-\str\ad_{\mathfrak r}(\Cas_{\mathfrak r}))(1\otimes 1).
\end{equation}
\end{theorem}
\begin{remark}
In \cite{Ko2}, Kostant proved the identity (\ref{eq 5.7}) when $\mathfrak {g}$ is a quadratic Lie algebra and $\mathfrak {r}$ is quadratic subalgebra of $\mathfrak {g}$. In fact, for the Lie algebraic case, the formula of $(D({\mathfrak g},{\mathfrak r}))^2$ in terms of Casimir elements goes back to Parthasarathy. In \cite{Par}, he obtained the formula under the assumption that ${\mathfrak g}={\mathfrak r}\oplus {\mathfrak p}$ is a Cartan decomposition and $\rank({\mathfrak g})=\rank({\mathfrak r})$. For this case, $\phi_\mathfrak{p}=0$. The Dirac operators for Lie superalgebras have been studied by several groups of researchers (\cite{HP,KaMP1,KaMP2,La,Pe}). In \cite{HP}, Huang and Pand\v{z}i\'{c} proved the identity (\ref{eq 5.7}) for the case ${\mathfrak r}=\mathfrak{g}_{\bar 0}$. In \cite{Pe}, Pengpan constructed the cubic Dirac operator for both full Lie superalgebra ${\mathfrak g}$ and its equal rank embeddings, where ${\mathfrak g}$ is a basic Lie superalgebra. The author also derived a formula for the square of Dirac operators (see the formula (40) in \cite{Pe}). In the infinite dimensional Lie superalgebras case, Landweber studied an affine analogue of the cubic Dirac operator (\cite{La}) for loop algebras, which was introduced much earlier by Kac and Todorov in \cite{KaT} on unitary representations of Neveu-Schwarz and Ramond superalgebras, and was studied further by Kac, M\"{o}seneder Frajria and Papi in \cite{KaMP1}.
\end{remark}

\section*{Acknowledgments}
This work is supported in part by National Natural Science Foundation of China (no. 11571182) and the talents foundation of Central South University of Forestry and Technology (104-0089). The authors thank Professor F.H. Zhu for valuable discussions, and the first author thanks Professor C.M. Bai for the invitation to visit Chern Institute of Mathematics since part of this work was done during the visit.

\end{document}